\documentclass[12pt]{amsart}

\usepackage{verbatim}
\usepackage{mathrsfs,amsfonts,amsmath,amssymb,epsfig,amscd,xy,amsthm,pb-diagram} 
\usepackage{tikz-cd}
\usepackage{geometry}
\usepackage[pdftex,hyperindex]{hyperref}
\usepackage{mathtools}

\numberwithin{equation}{section}       

\newtheorem{theorem}{Theorem}[section]
\newtheorem{proposition}[theorem]{Proposition}

\newtheorem{lemma}[theorem]{Lemma}

\newtheorem{corollary}[theorem]{Corollary}

\theoremstyle{definition}
\newtheorem{definition}[theorem]{Definition}
\theoremstyle{plain}

\theoremstyle{remark}

\newtheorem{remark}[theorem]{Remark}

\newtheorem{question}[theorem]{Question}

\newcommand{\C}{{\mathbb C}}
\newcommand{\Q}{{\mathbb Q}}
\newcommand{\R}{{\mathbb R}}
\newcommand{\Z}{\mathbb Z}

\DeclareMathOperator{\mymod}{\,mod}

\DeclareMathOperator{\re}{Re}

\newcommand{\bP}{{\mathbb P}}
\newcommand{\cO}{\mathcal{O}}
\newcommand{\cP}{\mathcal{P}}

\newcommand{\norm}[2][{}]{\left\|#2\right\|_{#1}}

\newcommand{\abs}[2][{}]{\left|#2\right|_{#1}}

\newcommand{\pair}[2]{\left\langle #1,#2 \right\rangle}
\newcommand{\proofof}[1]{\noindent{\bf Proof of #1. }}
\newcommand{\tto}{\dashrightarrow}

\newcommand{\suppress}[1]{{\bf\textcolor{red}{TEXT SUPPRESSED HERE.}}}

\newcommand{\mats}{\mathop{\mathrm{Mat}}}
\newcommand{\cvgts}{\mathop{\mathrm{Cvgt}}(\theta)}
\newcommand{\bbq}{\bar{\Q}}
\newcommand{\iprod}[2]{\left\langle{#1,#2}\right\rangle}
\newcommand{\mc}{\mathcal}
\newcommand{\supp}{\mathop{\mathrm{supp}}}

\newcommand{\cone}{k_1} 
\newcommand{\ctwo}{k_2} 
\newcommand{\firstgap}{p(n)} 
\newcommand{\secondgap}{q(n)} 
\newcommand{\rhoconst}{k_0} 

\newcommand{\inv}{g}                
\newcommand{\momap}{h}              
\newcommand{\crit}{{\mathrm{Crit}}} 
\newcommand{\ind}{\mathrm{Ind}}     
\newcommand{\torus}{\mathbb{T}}     
\newcommand{\bk}{\Bbbk}             
\newcommand{\ddeg}{\lambda}         
\renewcommand{\dim}{d}              
\newcommand{\nn}{N}                 
\newcommand{\op}{A}                 
\newcommand{\maxeval}{\xi_{\max}}    
\newcommand{\bmaxeval}{\bar{\xi}_{\max}}  
\newcommand{\evals}{\xi}            
\newcommand{\eval}{\xi}             
\newcommand{\rvecs}{\mathcal{U}}    
\newcommand{\rvec}{u}               
\newcommand{\cvecs}{\mathcal{V}}    
\newcommand{\cvec}{v}               
\newcommand{\cmax}{\cvec_{\max}}         
\newcommand{\bcmax}{{\bar\cvec}_{\max}}   
\newcommand{\rmax}{\rvec_{\max}}         
\newcommand{\brmax}{{\bar\rvec}_{\max}}  
\newcommand{\fnal}{\Psi}            
\newcommand{\discty}{\mathrm{Dsc}}  
\newcommand{\wt}{\psi}              
\newcommand{\dvec}{w}               
\newcommand{\dvecs}{\mathcal{W}}    
\newcommand{\rfn}{\sigma}           
\newcommand{\imunit}{{\mathsf{i}}}
\newcommand{\SL}{\operatorname{SL}}
\newcommand{\place}{\nu}            
\newcommand{\val}{v}                   

\title{Birational maps with transcendental dynamical degree}

\author{Jason P. Bell}
\address{Department of Pure Mathematics\\
University of Waterloo\\
Waterloo, ON N2L 3G1\\
Canada}
\email{jpbell@uwaterloo.ca}

\author{Jeffrey Diller}
\address{Department of Mathematics\\
  University of Notre Dame\\
  Notre Dame, IN 46556\\
  USA}
\email{diller.1@nd.edu}

\author{Mattias Jonsson}
\address{Dept of Mathematics\\
  University of Michigan\\
  Ann Arbor, MI 48109-1043\\
  USA}
\email{mattiasj@umich.edu}

\author{Holly Krieger}
\address{Department of Pure Mathematics and Mathematical Statistics\\ University of Cambridge\\ Cambridge CB3 0WB\\ UK}
\email{hkrieger@dpmms.cam.ac.uk}
\begin{document}
\begin{abstract}
  We give examples of birational selfmaps of $\mathbb{P}^d, d \geq 3,$ whose dynamical degree is a transcendental number. This contradicts a conjecture by Bellon and Viallet. The proof uses a combination of techniques from algebraic dynamics and diophantine approximation.
\end{abstract}

\maketitle

\setcounter{tocdepth}{1}
\tableofcontents

\section{Introduction}
\label{sec:intro}
The \emph{first dynamical degree} of a rational map $f\colon\bP^\dim\tto\bP^\dim$ is the quantity
$$
\ddeg(f) := \lim_{n\to\infty} \deg(f^n)^{1/n},
$$
where $f^n$ denotes the $n$th iterate of $f$, and $\deg(f^n) := \deg f^{-n}(H)$ is the preimage of a general hyperplane $H\subset\bP^\dim$.  The limit defining $\ddeg(f)$ always exists, and its value is a fundamental invariant for the dynamics of $f$.  For many rational maps, one has that $\ddeg(f) = \deg(f)$ is an integer.  In many other situations, it is known that $\ddeg(f)$ is the largest eigenvalue of some integer matrix.  It is also known~\cite{BF00,Ure18} that the first dynamical degree ranges through only countably many possible values in general.  

The values are not, however, limited to roots of integer polynomials.  In~\cite{BDJ20} the first three authors presented examples of rational self-maps $f\colon\bP^2\tto\bP^2$ whose first dynamical degrees are transcendental.  These examples are non-invertible.  For many purposes, both theoretical and applied, it is more natural to consider invertible dynamical systems.  However, the lack of invertibility in dimension two is essential to produce examples of self-maps with transcendental first dynamical degree, since~\cite{DiFa01} shows that the first dynamical degree of a birational surface map is always an algebraic integer; see also~\cite{BlCa16}. The same is true for polynomial automorphisms of $\mathbb{A}^3$ in characteristic zero~\cite{DaFa21}.

In fact, it was conjectured in~\cite{BeVi98} that the first dynamical degree of a birational map is always algebraic. Here we resolve that conjecture in the negative.  Specifically, we show that there are birational maps $f\colon\bP^\dim\tto\bP^\dim$, $\dim\geq 3$, whose first dynamical degrees are transcendental.

While we build on the methods introduced in~\cite{BDJ20}, we take a different approach to deriving the crucial power series formula for the dynamical degrees of our maps, and we obtain a substantially more general transcendence result.  The list of examples we obtain is infinite, but not completely explicit, and there remain some very interesting further questions.
As in~\cite{BDJ20}, our examples are based on \emph{monomial maps}, i.e.\ maps $\momap_\op\colon\bP^\dim\tto\bP^\dim$ whose components $\momap_{\op,j} = x_1^{a_{j1}} \dots x_\dim^{a_{j\dim}}$ are monomials with exponents specified by the $j$th row of a $\dim\times \dim$ integer matrix $\op$.  Since we
aim to construct birational maps, we will always take $\op\in {\rm SL}_\dim(\Z)$.  Our main theorem may be stated as follows.\footnote{J.~Blanc informed us that he has independently been able to modify the construction of~\cite{BDJ20} to obtain birational maps with dynamical degrees satisfying a power series formula similar to \eqref{eqn:pwrseriesformula}.  These maps might also serve to produce transcendental examples.}

\begin{theorem}
\label{thm:mainapp} For each $d\geq 3$, there exists a birational involution $\inv\colon\bP^\dim\tto\bP^\dim$ and matrices $\op\in {\rm SL}_\dim(\Z)$ such that the birational maps $f\colon\bP^\dim\tto\bP^\dim$ given by 
\begin{equation}
\label{eqn:fformula}
f = \inv\circ\momap_\op
\end{equation}
have transcendental dynamical degree $\ddeg(f)$.
\end{theorem}

All maps in the theorem have coefficients in $\{-1,0,1\}$, so the field of definition for $f$ can be taken to be any field of characteristic different from $2$. The involution $g$ is explicit, given at the beginning of \S\ref{sec:involution},
but the matrix $\op$ is not.  As we will explain in more detail shortly, we begin with a suitable particular element of ${\rm SL}_\dim(\Z)$, and take $\op$ to be a large enough power of a fairly general conjugate of of this element.
However, at the end of this article in \S\ref{sec:specific} we explain how one can check, with some computer assistance, that the conclusion of Theorem~\ref{thm:mainapp} applies for a particular matrix $\op$.  See~\eqref{eqn:eg} for the precise matrix we consider.

To compute $\deg(f^n)$ for the maps $f$ in Theorem~\ref{thm:mainapp}, we use that by duality, $\deg(f^n)$ is also equal to the intersection number between a fixed hyperplane $H\subset\bP^\dim$ and the forward image $f^n(\ell)$ of a general line $\ell$.  As we explain in \S\ref{sec:geometry}, it is convenient for tracking the successive images of $\ell$ to regard iterates of $f$ as maps between various toric blowups of $\bP^\dim$.  We then show in \S\ref{sec:bregs} that for suitable $\op$, the dynamical degree of $f$ satisfies an equation involving a power series with integer coefficients.  

To state the precise formula, we introduce some notation.  Let $\dim$ be a positive integer, and $\rvecs,\cvecs\subset \Z^\dim$ finite, non-empty sets of non-zero vectors. For any $\dim\times\dim$ integer matrix $\op$, we set
$$
\fnal_{\rvecs,\cvecs}(\op) := \sum_{\cvec\in\cvecs}\max_{\rvec\in\rvecs}\pair{\rvec}{\op \cvec},
$$
where $\langle\cdot,\cdot\rangle$ denotes the standard bilinear pairing on $\Z^\dim$.
The resulting integer-valued function is piecewise linear in the entries of $\op$.  The main result of \S\ref{sec:bregs} is as follows.

\begin{theorem}
\label{thm:degformula} Suppose that $\tilde \op\in {\rm SL}_\dim(\Z)$ has irreducible characteristic polynomial and eigenvalues of largest magnitude equal to a complex conjugate pair $\maxeval,\bmaxeval$ with $\maxeval^j\notin\R$ for any non-zero $j\in\Z$.  If $\op = \tilde \op^N$ for large enough $N\in\Z_{\geq 0}$, and $f$ is given by~\eqref{eqn:fformula}, then 
$\ddeg = \ddeg(f)$ satisfies
\begin{equation}
\label{eqn:pwrseriesformula}
\sum_{j=1}^\infty 
\fnal_{\rvecs,\cvecs}(\op^j)\ddeg^{-j} = 1,
\end{equation}
where $\rvecs,\cvecs\subset\Z^\dim$ are finite sets of vectors that depend only on the dimension $\dim$.
\end{theorem}

The particular sets $\rvecs,\cvecs$ referred to in this theorem are given in~\eqref{eq:wt} and~\eqref{eq:cvecs}.  Regardless, for any fixed $\rvec\in\Z^\dim$, the sequence $\sum_{\cvec\in\cvecs} \pair{\rvec}{\op^j \cvec}$ is an integer linear recurrence, so if $\rvecs$ contained only one element, the power series~\eqref{eqn:pwrseriesformula} would define a rational function with integer coefficients, and it would follow that $\ddeg$ is algebraic.  For $\rvecs$ as given, however, the coefficients $\fnal_{\rvecs,\cvecs}(\op^j)$ are obtained by maximizing over several integer linear recurrences.   The condition on the leading eigenvalue of $\op$ guarantees that the largest among them varies irregularly as $j$ increases.  Under these circumstances it would seem difficult for $\ddeg$ and the value of the series to be simultaneously algebraic.  The following result solidifies this intuition.  Together with Theorem~\ref{thm:degformula}, it suffices for establishing Theorem~\ref{thm:mainapp}. 
Here we say that $z,w\in\C$ have an angular resonance if $z^a\bar w^b \in \R$ for some integers $a,b>0$.

\begin{theorem}
\label{thm:main}
Let $\tilde \op \in {\rm SL}_\dim(\Z)$ be a matrix with irreducible characteristic polynomial.  Suppose that there are no angular resonances between distinct eigenvalues of $\tilde \op$ and that the eigenvalues of largest magnitude are a complex conjugate pair $\maxeval$, $\bmaxeval$.
Then, for any finite sets $\rvecs,\cvecs\subset\Z^\dim\setminus\{0\}$ with $\#\rvecs\geq 2$, there exist matrices $\op\subset {\rm SL}_\dim(\Z)$ conjugate to $\tilde \op$ such that 
\begin{equation}
\label{eqn:theseries}
\sum_{j=1}^\infty \fnal_{\rvecs,\cvecs}(\op^{Nj}) x^j
\end{equation}
is transcendental for any $N\geq 1$ and any real $x\in\bbq\cap (0,|\maxeval|^{-N})$.
\end{theorem}

Note that there is no angular resonance between the leading
eigenvalues $\maxeval$ and $\bmaxeval$ if and only if
$\maxeval^j\notin \R$ for any non-zero $j\in\Z$.  When $\dim=3$ the
remaining eigenvalue is real, so this is the entire content of the no
angular resonance requirement.  When $\dim>3$, the requirement is more
restrictive and implies in particular that $\op$ has at most one real
eigenvalue.  Note also that the condition on $x$ implies that it
belongs to the domain of convergence of the series
\eqref{eqn:theseries}.  Indeed it follows from Corollary \ref{cor:xiandbarxi} below and irreducibility of the characteristic polynomial of $\tilde\op$ that the radius of convergence of the series is \emph{exactly} $|\maxeval|^{-N}$, though we only need to know that that it is at least this large. 

To prove Theorem~\ref{thm:main} we note that the dynamics of the linear map $\op$ on $\Z^\dim\subset\C^\dim$ can be understood by diagonalizing $\op$.  Write $\maxeval = |\maxeval|e^{2\pi\imunit\theta}$, where the normalized argument $\theta\in\R$ is irrational by hypothesis on $\maxeval$.  Using our assumptions on the spectrum $(\eval_1=\maxeval,\eval_2=\bmaxeval,\eval_3,\dots,\eval_\dim)$ of $\op$, we show that for large enough $j\in\Z_{\geq 0}$, we have
$$
\fnal_{\rvecs,\cvecs}(\op^j) = \pair{\gamma(j\theta)}{(\eval_1^j,\ldots ,\eval_\dim^j)},
$$
for some $1$-periodic, piecewise constant function $\gamma\colon\R\to \bbq^\dim$. The hypotheses on $A$, $\rvecs$, $\cvecs$ imply, however, that $\gamma$ is not (globally) constant.

Theorem~\ref{thm:main} then reduces to the following theorem, which we prove in \S\ref{sec:proof}. Note here
that we rely implicitly on a fixed embedding $\bbq\subset\C$ and its associated archimedean absolute value $|\cdot|$.

\begin{theorem}
\label{thm:mainb}
Let $\theta\in\R$ be an irrational number and $\rho \in \bbq^\dim$ be a vector whose coordinates each satisfy $|\rho_j|<1$ and are pairwise multiplicatively
independent.  Let $\gamma\colon\R\to\bbq^\dim$ be a non-constant but piecewise constant, $1$-periodic function with (discrete) discontinuity set $\discty(\gamma)\ne\emptyset$ such that
\begin{itemize}
 \item (discordance) for any $t,t' \in \discty(\gamma)\cup\{0\}$ and any $a,b\in\Z$, $a\theta = b(t-t')\mymod 1$ implies $a=0$ and either $t-t'\in\Z$ or $b$ is even;
 \item (maximality) for all sufficiently large integers $j$, the function $$t\mapsto \pair{\gamma(t)}{(\rho_1^j,\ldots ,\rho_\dim^j)}$$ is real-valued, non-constant,  and maximized by $t=j\theta$.
 \end{itemize}
Then 
$$
\Omega:=\sum_{j=1}^\infty \pair{\gamma(j\theta)}{(\rho_1^j,\ldots ,\rho_\dim^j)}
$$ is transcendental.
\end{theorem}

The proof of Theorem~\ref{thm:mainb} expands on ideas from~\cite{BDJ20}.  In particular, we rely heavily on a lower bound (Theorem~\ref{thm:Evertse}) for Diophantine approximations  due to Evertse~\cite{Eve84} and a finiteness result (Theorem~\ref{thm:unitequations}) for solutions of unit equations due to Evertse, Schlickewei and Schmidt~\cite{ESS02}.  The bulk of the proof consists of carefully analyzing the continued fraction expansion of $\theta$ to identify and exploit large, but necessarily finite, stretches in the series defining $\Omega$ in which the coefficients satisfy some sort of linear recurrence.

To pass from Theorem~\ref{thm:mainb} to Theorem~\ref{thm:main}, we set $\rho_i = (x\eval_i)^{N}$ for $i=1,\ldots ,\dim$.  Since there are no angular resonances among the $\eval_i$, the resulting $\rho_i$ are multiplicatively independent for any $x$ and $N$.  When derived from the data in Theorem~\ref{thm:mainb}, the function $\gamma=\gamma_\op$, and especially its discontinuity set $\discty(\gamma_\op)$, depend on the sets $\rvecs$ and $\cvecs$ and the matrix $\op$.  The maximality condition in Theorem~\ref{thm:main} holds for any choice of $\rvecs$, $\cvecs$ and $\op$.  The reason for replacing the given matrix $\tilde \op$ in Theorem~\ref{thm:main} with a conjugate matrix $\op$ is to guarantee that the discordance hypothesis is also satisfied.  

Our approach to finding suitable conjugates relies on the fact that all powers $\maxeval^j$ of the maximal eigenvalue in Theorem~\ref{thm:main} lie in the unit subgroup $\cO_K^*$ of the integers $\cO_K$ in the number field $K$ generated by eigenvalues of $\op$.  On the other hand, the elements of $\discty(\gamma_\op)$ are normalized arguments of elements of $K$, but these elements need not be units.  In fact, given a specific matrix $\tilde \op\in {\rm SL}_\dim(\Z)$, it is not difficult to find a specific conjugate $\op$ by trial and error and then check by computer algebra that no element (or difference between elements) of $\discty(\gamma_\op)$ is the normalized argument of an algebraic unit.  We account for this phenomenon by showing that suitable conjugates of $\tilde \op$ are in some sense generic.  See Theorem~\ref{thm:noresonances} and its proof in \S\ref{sec:resonances}.  The argument there relies on the general fact (Lemma~\ref{lem:nonconst2}) that a non-constant rational function $\tau\in K(x)$ cannot have range $\tau(K)$ contained in the group of units $\cO_K^*$.

Since the discordance hypothesis of Theorem~\ref{thm:mainb} is a bit unnatural and difficult to arrange, it is worth stressing that it is needed only when the irrational number $\theta$ is \emph{badly approximable} (equivalently, \emph{of bounded type}), i.e.\ when the continued fraction expansion of $\theta$ has uniformly bounded coefficients.  For well approximable $\theta$, the proof of Theorem~\ref{thm:mainb} is substantially simpler, effectively ending with Corollary~\ref{cor:nobiggaps} rather than the subsequent and more technical arguments of \S\ref{sec:nonvanishing} and \S\ref{sec:fixthisname}. Nor in this case do we need Theorem~\ref{thm:noresonances}.  Unfortunately, however, it is unclear to us whether/when the normalized argument $\theta$ of the leading eigenvalue $\maxeval$ in Theorems~\ref{thm:degformula} and~\ref{thm:main} is well approximable.  

\begin{question}
Are there algebraic units whose normalized arguments $\theta$ are irrational and well approximable?
Likewise, are there any for which $\theta$ is irrational and badly approximable?
\end{question}
Let us close by returning to the first paragraph of this introduction and the dynamical significance of the first dynamical degree.  The interested reader may consult~\cite{BDJ20} for a longer account, but here we recall a single aspect of that discussion.  The first dynamical degree is only one of $\dim-1$ intermediate dynamical degrees for a birational map $f\colon\bP^\dim\tto \bP^\dim$; see~\cite{DS05a,Tru20,Dan20}. For rational maps $f$ over $\C$, the logarithm of the largest of these dynamical degrees is known~\cite{DS05a} to be an upper bound for the entropy of $f$ (see also~\cite{FTX22} for a non-archimedean version),
and in many instances~\cite{BS92,BD01,Gue05,DS05b,Duj06,Vig14}, the two quantities are known to be equal.  Hence it is interesting to ask whether the first dynamical degrees of the maps we construct here are also the largest.
\begin{question}
Does there exist a birational map $f\colon\bP^\dim\tto\bP^\dim$ for which $\ddeg_1(f)$ is transcendental and also maximal among intermediate dynamical degrees $\ddeg_i(f)$, $i=1,\dots,\dim-1$?  Does there exist a birational map $f\colon\bP^\dim\tto\bP^\dim$ for which \emph{all} intermediate dynamical degrees are transcendental? 
\end{question}
The outline of the paper is as follows.  In \S\ref{sec:geometry}, we give background on toric threefolds and monomial maps.  This is used in \S\ref{sec:bregs}, where we analyze the maps $\inv$ and $f=\inv\circ\momap_A$ and prove Theorem~\ref{thm:degformula}.  In \S\ref{sec:sunit}, we give the relevant background from Diophantine approximation, which will be used in the proofs of Theorems~\ref{thm:main} and~\ref{thm:mainb}.  In \S\ref{sec:prelim} we prove Theorem~\ref{thm:main}, or more precisely reduce it to Theorem~\ref{thm:mainb}, which is proved in \S\ref{sec:proof}. Finally, in~\S\ref{sec:general} we complete the proof of Theorem \ref{thm:mainapp}, the main remaining step being to construct suitable characteristic polynomials from which to obtain our matrices $A$.  In \S\ref{sec:specific}, we focus on a specific matrix $A\in\SL_3(\Z)$ to explain how one can use computer algebra to certify that particular maps $f = \inv\circ \momap_A$ satisfy the conclusion of Theorem~\ref{thm:mainapp}.

\subsection*{Acknowledgments} 
We thank Nguyen-Bac Dang for his thoughtful comments about this article.  We would also like to thank the anonymous referee for their many useful suggestions and careful reading of this article.

The first author was partially supported by NSERC grant RGPIN-2016-03632; the second author by NSF grant DMS-1954335; the third author by NSF grants DMS-1600011 and DMS-1900025, and the United States-Israel Binational Science Foundation; and the fourth author by Isaac Newton Trust (RG74916).

\section{Intersection theory, toric varieties and monomial maps}
\label{sec:geometry}
We work over an algebraically closed field $\bk$ of characteristic different from two.

\subsection{Rational maps and intersection numbers}
\label{sec:intersect}  We begin with a somewhat ad hoc definition of intersection numbers between curves and divisors, consistent with the general theory of~\cite{Ful84}.  Let $X$ be a smooth proper variety of dimension $\dim\geq 2$, let $\mathsf{C}\subset X$ an irreducible curve, and let $D$ be a Cartier divisor on $X$.  Consider the inclusion map $\iota\colon  \mathsf{C}\to X$ and normalization map $\tau\colon\tilde{\mathsf{C}}\to\mathsf{C}$.
In this situation we define
\begin{equation*}
  (\mathsf{C}\cdot D):=\deg(\tau^*\iota^*\cO_X(D)),
\end{equation*}
the degree of the line bundle $\tau^*\iota^*\cO_X(D)$ on $\tilde{\mathsf{C}}$.
When $\mathsf{C}$ and $D$ are smooth, and $\mathsf{C}$ is not contained in the support of $D$, $(\mathsf{C}\cdot D)$ is the number of points in $\mathsf{C}\cap D$ counted with multiplicity.  Note also that $(\mathsf{C}\cdot D)$ only depends on the linear equivalence class of $D$.  When $X=\bP^d$, we have $(\mathsf{L}\cdot D)=\deg D$ for every line $\mathsf{L}$.

\medskip
Now consider a birational map $f\colon X_1\dashrightarrow X_2$ between smooth varieties. The \emph{indeterminacy set} $\ind(f)$ is the smallest set such that
$f\colon X_1\setminus \ind(f)\to X_2$ is a morphism; this set has codimension
at least two. The \emph{critical set} $\crit(f)$ is the (finite) union of all
irreducible hypersurfaces contracted by $f$.  For any irreducible subvariety $V\subset X_1$ not contained in $\ind(f)$, we adopt the convention that $f(V) := \overline{f(V)\setminus \ind(f)}$ is the proper transform of $V$ by $f$.  In particular $f(V)$ is irreducible and, if not contained in $\crit(f)$, of the same dimension as $V$. 

Let $D_2$ be a divisor on $X_2$. The pullback $f^*D_2$ is the divisor on $X_1$ defined as follows. Let $X\subset X_1\times X_2$ be the Zariski closure of the graph of $f$, and $\pi_j\colon X\to X_j$, $j=1,2$, the projections. Then $f^*D_2:=\pi_{1*}\pi_2^*D_2$, where we pull back $D_2$ as a Cartier divisor, then push forward $\pi_2^*D_2$ as a Weil divisor.  Since $X_1$ and $X_2$ are smooth any Weil divisor on either $X_j$ is also Cartier. We rely on the following version of the projection formula that is easily verified.

\begin{proposition}
\label{prop:pushpull}
In the situation above, let $\mathsf{C}_1\subset X_1$ be an irreducible curve disjoint from
$\ind(f)$ and not contained in $\crit(f)$, and let $D_2$ be a Cartier divisor on $X_2$.
Then
\begin{equation}\label{e101}
  (\mathsf{C}_1\cdot f^*D_2)=(f(\mathsf{C}_1)\cdot D_2).
\end{equation}
\end{proposition}
We are particularly interested in the case $X_1 = X_2 = \bP^\dim$. In homogeneous coordinates $f$ is given by $$[x_0,\dots,x_\dim]\mapsto[f_0,\dots,f_\dim],$$ where the $f_j$ are homogeneous polynomials, all of the same degree and without common factors.  The \emph{(algebraic) degree} of $f$ is then defined to be $\deg f := \deg f_j = \deg f^* H = (\mathsf{L}\cdot f^* H)$, where $H\subset\bP^\dim$ is any hyperplane and $\mathsf{L}$ is any line.  So Proposition \ref{prop:pushpull} allows us to rewrite $\deg f$ as follows.

\begin{corollary}
\label{cor:degree}
The degree of a birational map $f\colon\bP^\dim\dashrightarrow\bP^\dim$ is given by
$$
\deg f=(f(\mathsf{L})\cdot H),
$$
where $\mathsf{L}\subset \bP^\dim$ is any line disjoint from $\ind(f)$ and not contained in $\crit(f)$, and $H\subset\bP^\dim$ is any hyperplane.
\end{corollary}

\noindent Since $\ind(f)$ has codimension at least two, the hypothesis of the corollary is satisfied by a general line $\mathsf{L}$, i.e.\ a line corresponding to a general point in the Grassmannian $\mathrm{Gr}(2,\dim+1)$.

If $f,g\colon\bP^\dim\tto\bP^\dim$ are rational maps, then $\deg(f\circ g) \leq (\deg f)(\deg g)$.  This fact implies that the limit in the following definition exists.

\begin{definition} 
The \emph{first dynamical degree} of a rational map $f\colon\bP^\dim\tto\bP^\dim$ is the quantity
$\ddeg(f) := \lim_{n\to\infty} (\deg f^n)^{1/n}$.
\end{definition}

\subsection{Toric varieties}
\label{sec:toric}

For our purposes, a \emph{toric variety} will be a smooth algebraic compactification $X$ of the torus $\torus :=\mathbb{G}_m^\dim$ such that the natural action of the torus on itself extends to all of $X$. Any toric variety is defined by a lattice $\nn\cong\Z^\dim$ and a fan $\Sigma(X)$ in $\nn$, by which is meant
a collection of regular rational simplicial cones inside $\nn\otimes\R\cong\R^\dim$, satisfying natural axioms, see~\cite{Ful93}.

In what follows, we fix a basis for $N\simeq\Z^\dim$. The fan of $\bP^\dim$ is then the set of cones in $\R^\dim$ 
generated by the proper subsets of
\begin{equation}
\label{eqn:p3fan}
\cP=\{(-1,\dots,-1),(1,0,\dots,0),(0,1,0,\dots,0),\dots,(0,\dots,0,1)\}.
\end{equation}

For any toric variety, the complement $X\setminus \torus$ is a simple normal crossings divisor, and for $k>0$, the $k$-dimensional cones in $\Sigma$ correspond to $\torus$-invariant irreducible subvarieties of $X\setminus \torus$ of codimension $k$.  In particular, rays of $\Sigma(X)$ correspond to irreducible hypersurfaces $E\subset X\setminus \torus$ which we will call \emph{poles}.\footnote{They are in fact the (simple) poles in $X$ of the $\torus$-invariant form $\frac{dy_1\wedge\dots\wedge dy_{\dim}}{y_1\dots y_{\dim}}$.  It is standard to call $E$ a `torus invariant hypersurface', but we find the shorter term convenient.}

We let $\cvec_E\in \nn$ denote the unique primitive element (i.e.\  $\cvec_E\ne 0$ and $\cvec_E\not\in bN$ for $b\ge 2$) in the ray in $\Sigma(X)$ corresponding to a pole $E\subset X$. Given a primitive element $\cvec\in N$, we say that $X$ \emph{realizes} $\cvec$ if $\cvec=\cvec_E$ for some pole $E$ of $X$.
For example, the poles of $\bP^\dim$ corresponding to the elements of $\cP$ above are the coordinate hyperplanes $\{x_j=0\}$, $0\le j\le\dim$ in $\bP^\dim$.

Each pole $E$ of $X$ is itself a toric variety of dimension $\dim-1$, its toric structure defined by a natural fan in the quotient lattice $N/\Z \cvec_E$, and its torus $\torus_E\simeq\mathbb{G}_m^{\dim-1}$ concretely realized as the set of points in $E$ not contained in any other pole.

We set $M:=\operatorname{Hom}(N,\Z)$ and write $\langle \rvec,\cvec\rangle\in\Z$ for the pairing between $\rvec\in M$ and $\cvec\in N$. The elements of $M$ can be identified with the set of characters $\torus\to\mathbb{G}_m$. The identification $N\simeq\Z^\dim$ induces an identification $M\simeq\Z^\dim$, and the characters associated to the standard basis vectors of $\Z^\dim$ serve as coordinates $(y_1,\dots,y_\dim)$ on $\torus$, giving an isomorphism $\torus{\overset\sim\to}\mathbb{G}_m^\dim$. The character associated to $(a_1,\dots,a_\dim)\in M$ is then the monomial $y_1^{a_1}\cdots y_\dim^{a_\dim}$.

A \emph{toric modification} is a birational morphism $\pi\colon\hat X \to X$ between toric varieties that restricts to the identity on $\torus$.  The fan $\Sigma(\hat X)$ is then a simplicial subdivision of $\Sigma(X)$; and for each pole $E\subset \hat X$ contracted by $\pi$, the image $\pi(E)$ is equal to the intersection of two or more poles in $X$.
Given any two toric varieties $X, X'$, there exists a third $\hat X$ that modifies both of them.
Moreover, given any toric variety $X$ and a primitive element $\cvec\in N$, there exists a toric modification $\hat X\to X$ such that $\hat X$ realizes $\cvec$.

A divisor $D$ supported on poles of $X$ may be encoded by a support function $\wt_D\colon \nn\to\Z$ given by setting $\wt_D(\cvec_E)$ equal to the coefficient of $E$ for each pole $E\subset X$ and then extending linearly across each cone in $\Sigma(X)$.
If $\pi\colon\hat X\to X$ is a toric modification, then $\wt_{\pi^*D}=\wt_D$.  Moreover, $D$ is principal if and only if $\wt_D$ is linear.

For instance, the hyperplane at infinity $\{x_0=0\}$ on $\bP^\dim$ has support function 
\begin{equation}\label{eq:wt}
\wt(v) = \max_{\rvec\in\rvecs} \langle\rvec,\cvec\rangle,
\end{equation}
where $\rvecs\subset M:=\operatorname{Hom}(N,\Z)\simeq\Z^\dim$ is given by
\begin{equation}\label{eq:rvecs}
\rvecs := \{(0,\dots,0),(-1,0,\dots,0),(0,-1,0,\dots,0),\dots(0,\dots,0,-1)\}.
\end{equation}

\begin{definition}[see e.g. \cite{GHK15}] 
An irreducible curve $\mathsf{C}$ in a toric variety $X$ is \emph{internal} if $\mathsf{C}\cap\torus\neq \emptyset$. We say that \emph{$X$ is adapted to $\mathsf{C}$} if for each pole $E\subset X$, the intersection $\mathsf{C}\cap E$ is contained in $\torus_E$.
\end{definition}

If $\hat X\to X$ is a toric modification then we identify any internal curve $\mathsf{C}$ in $X$ with its (still internal) proper transform on $\hat X$.  If $X$ is adapted to $\mathsf{C}$, so is $\hat X$.  Moreover, we have

\begin{proposition}
For any internal curve $\mathsf{C}$ on a toric variety $X$, there is a toric modification $\pi\colon\hat X\to X$ such that $\hat X$ is adapted to $\mathsf{C}$.  
\end{proposition}

\begin{proof}
If $X$ is not adapted to $\mathsf{C}$, then there are poles $E_1,\dots,E_m\subset X$, with $m\ge2$, such that $\mathsf{C}\cap E_1\cap\dots\cap E_m \neq \emptyset$.  The blowup $\pi\colon\tilde X\to X$ of $\bigcap_{j=1}^mE_j$ is toric and contracts a pole $\tilde E_0\subset \tilde X$ that meets the proper transform $\tilde{\mathsf{C}}\subset \tilde X$ of $\mathsf{C}$.  Additionally, if $\tilde E_j\subset \tilde X$ denotes the proper transform of $E_j$, then $\pi^* E_j = \tilde E_j + \tilde E_0$.  So if $D=\sum_{E\subset X} E$ and $\tilde D = \sum_{\tilde E\subset \tilde X} \tilde E$ are the reduced divisors with supports equal to all poles of $X$ and $\tilde X$, then $\pi^* D = \tilde D + k\tilde E_0$ for some $k\geq 1$.  From this and Proposition \ref{prop:pushpull}, we get
$$
(\mathsf{C}\cdot D) = (\tilde{\mathsf{C}}\cdot \pi^*D) = (\tilde{\mathsf{C}}\cdot \tilde D) + k(\tilde{\mathsf{C}}\cdot \tilde E_0)
> (\tilde{\mathsf{C}}\cdot \tilde D) \geq 0,
$$
where the last inequality follows from the fact that $\tilde{\mathsf{C}}$ is an internal curve.  If $\tilde X$ is not adapted to $\tilde{\mathsf{C}}$, we repeat the above as often as necessary.  At each step the intersection between the set of poles and $\tilde{\mathsf{C}}$ drops by at least one.  Since the intersection must remain non-negative, the process must stop in finitely many steps, at which point $X$ is adapted to $\tilde{\mathsf{C}}$.  That is, if at this point $E_0$ is the (further) blowup of any other intersection between two or more poles, we must have $\tilde{\mathsf{C}}\cdot E_0 = 0$; hence $\tilde{\mathsf{C}}$ intersects at most one pole.
\end{proof}

In light of this discussion we can associate to any internal curve $\mathsf{C}$ the following measure on $\nn$:
\begin{equation}
\label{eq:mu}
\mu_{\mathsf{C}} := \sum (\mathsf{C}\cdot E)\delta_{\cvec_E},
\end{equation}
where $\delta_{\cvec_E}$ is the point mass supported at $\cvec_E\in \nn$, and the sum is over poles $E\subset X$ in some/any toric variety adapted to $\mathsf{C}$.
For instance, if $\mathsf{L}$ is a general line in $\bP^\dim$, then
$$
\mu_{\mathsf{L}} = \sum_{\cvec\in\cP} \delta_\cvec,
$$
with $\cP$ as in \eqref{eqn:p3fan}.
If $\mathsf{C}$ is an internal curve on a toric variety $X$ and $D$ is a divisor supported on poles of $X$, then it follows from~\eqref{e101} that the intersection number $(\mathsf{C}\cdot D)$ is unchanged by toric modifications $\pi\colon\hat X \to X$, i.e. $(\mathsf{C}\cdot D) = (\mathsf{C}\cdot \pi^* D)$.  Taking $\hat X$ adapted to $\mathsf{C}$, it therefore follows from definitions that
\begin{equation}\label{eq:intno}
(\mathsf{C}\cdot D) = \int_\nn \wt_D \,\mu_{\mathsf{C}} = \sum_{E\subset \hat X} (\mathsf{C}\cdot E) \wt_D(\cvec_E).
\end{equation}
For example, the degree of an internal curve $\mathsf{C}\subset\bP^\dim$ is given by 
$$
(\mathsf{C}\cdot\{x_0=0\})=\int_N\wt\,\mu_C,
$$
where $\wt$ is given by~\eqref{eq:wt}.

Since the intersection number with a principal divisor must vanish, we obtain
\begin{corollary}\label{cor:balanced}
The measure $\mu_{\mathsf{C}}$ associated to an internal curve $\mathsf{C}$ is \emph{balanced} in the sense that $\sum_{E\subset X} (\mathsf{C}\cdot E)\cvec_E = 0 \in \nn$
for any $X$ adapted to $\mathsf{C}$.
\end{corollary}

\begin{remark}
The measure $\mu_{\mathsf{C}}$ associated to an internal curve corresponds to the Minkowski weight, in the sense of~\cite{FS97}, for the class of the curve $\mathsf{C}$.
\end{remark}

\subsection{Monomial maps}
For monomial maps and their dynamics, see \cite{Fav03, HP07, JoWu11, FaWu12, Lin12}.
\label{sec:monomial}
\begin{definition}
Let $A = (a_{ij})_{1\le i,j\le\dim}$ be a $\dim\times\dim$ integer matrix with $\det A\neq 0$.  We call $\momap_A\colon \torus\to \torus$ given by
\begin{equation*}
\momap_A(y_1,\dots,y_\dim)
=(y_1,\dots,y_\dim)^A
:=(y_1^{a_{11}}\dots y_\dim^{a_{1\dim}},\dots,y_1^{a_{\dim1}}\dots y_\dim^{a_{\dim\dim}})
\end{equation*}
the \emph{monomial map} associated to $A$.
\end{definition}

In what follows we will always assume that $A\in {\rm GL}_\dim(\Z)$, i.e. $\det A = \pm 1$, in which case $\momap_A$ is an automorphism of $\torus$ and extends to a birational map $\momap_A\colon X\tto X'$ between any two $\dim$-dimensional toric varieties.

Our convention for monomial maps is that $A\in\operatorname{GL}(N)$, so the induced automorphism $M\to M$ is given by the transpose $A^T$.
Note that for any $n\in\Z$, we also have $\momap_A^n = \momap_{A^n}$.

\begin{proposition}
  \label{prop:mo}
  Suppose that $A\in {\rm GL}_\dim(\Z)$, that $X,X'$ are toric varieties, and that $\momap= \momap_A\colon X\tto X'$ is the associated monomial map. Assume that for every pole $E\subset X$ there exists a pole $E'\subset X'$ such that $A\cvec_E=\cvec_{E'}$.
Then $h(E)=E'$. Moreover $h$ is an isomorphism in a neighborhood of $\torus_E$, and sends $\torus_E$ onto $\torus_{E'}$.
In particular, $\crit(\momap)=\emptyset$.
\end{proposition}

\begin{proof}
It suffices to prove the statement about $\torus_E$ and $\torus_{E'}$.
Pick $\rvec'_1\in M$ such that $\langle \rvec'_1,\cvec_{E'}\rangle=1$, and elements $\rvec'_2,\dots,\rvec'_\dim\in M$ that generate the lattice $\cvec_{E'}^\perp:=\{u\in M\mid \langle u,v_{E'}\rangle=0\}$. Each $\rvec'_j$ defines a monomial $\chi'_j$ in $(y_1,\dots,y_\dim)$, and $\chi':=(\chi_1,\dots,\chi_\dim)$ gives a birational map of $X$ to $\mathbb{A}^\dim$ which is an isomorphism in a neighborhood of $\torus_{E'}$ and sends $\torus_{E'}$ onto the coordinate hyperplane $\{w_1=0\}$ in $\mathbb{A}^1\times\mathbb{G}_m^{\dim-1}$.

Set $\rvec_j=A^Tu'_j$ for $1\le l\le\dim$. Then $\langle u_1,\cvec_E\rangle=\langle u'_1,v_{E'}\rangle=1$ and $\rvec_2,\dots,u_\dim$ generate the lattice $\cvec_{E}^\perp$. Each $u_j$ defines a monomial $\chi_j$ and $\chi:=(\chi_1,\dots,\chi_\dim)$ defines a birational map of $X$ to $\mathbb{A}^\dim$ that is an isomorphism in a neighborhood of $\torus_{E}$ and sends $\torus_{E}$ onto the hyperplane $\{w_1=0\}$ in $\mathbb{A}^1\times\mathbb{G}_m^{\dim-1}$. By construction, $\chi=\chi'\circ h$, and the result follows.
\end{proof}

The image of any internal curve $\mathsf{C}\subset\torus$ under a monomial birational map is a new internal curve, and we have:

\begin{corollary}
\label{cor:mopush}
If $\mathsf{C}\subset\torus$ is an internal curve and $A\in\operatorname{GL}_\dim(\Z)$, then
\begin{equation}\label{eq:mopush}
  \mu_{\momap_A(\mathsf{C})} = A_*\mu_{\mathsf{C}},
\end{equation}
where $\mu_{\mathsf{C}}$ and $\mu_{\momap_A(\mathsf{C})}$ are the associated measures on $N\cong\Z^\dim$.
\end{corollary}

\begin{proof}
Since $A$ is invertible over $\Z$, it preserves the set of primitive vectors in $\Z^d$.  Hence
the formula follows from the previous proposition with $X$ adapted to $\mathsf{C}$ and $X'$ adapted to $\momap_A(\mathsf{C})$.  
\end{proof}

\begin{corollary}
\label{cor:modd}
For any $A\in\operatorname{GL}_\dim(\Z)$, the dynamical degree of the monomial map $\momap_A$ is equal to the absolute value of the leading eigenvalue(s) of $A$. 
\end{corollary}

\begin{proof}
Taking $\mathsf{C}=\mathsf{L}$ to be a general line in $\bP^\dim$ and integrating the function $\wt$ in~\eqref{eq:wt} against~\eqref{eq:mopush} , we obtain
\begin{equation}\label{eq:degmomap}
\deg\momap_A^n = \int \wt\,A^n_*\mu_{\mathsf{L}}.
\end{equation}
If we add a linear function to $\wt$, the integral does not change.  So replacing $\psi(\cvec)$ with e.g. $\psi(\cvec) +\langle \rvec,\cvec\rangle$, where $\rvec=\frac14(1,\dots,1)$, we may assume that $\norm{\cvec}\le\wt(\cvec)\le C\norm{\cvec}$ for some norm $\norm{\cdot}$ on $N\otimes\R$ and some constant $C>1$. Thus $\deg \momap_{A}^n$ is  multiplicatively comparable, uniformly in $n$, to $\max_{\cvec\in\supp\mu_{\mathsf{L}}} \norm{A^n\cvec}$.  Since the vectors in $\supp\mu_{\mathsf{L}}$ span $\nn$, we see further that for large $n$, 
$
\deg \momap_A^n
$
is comparable to $n^{k-1} \rho^n$, where $\rho$ is the magnitude of a leading eigenvalue for $A$ and $k$ is the size of the largest Jordan block for such an eigenvalue. Thus $\ddeg(\momap_A)=\rho$.
\end{proof}

\section{Degrees of certain birational maps}
\label{sec:bregs}
In this section we study the composition of a birational monomial map with a well chosen birational involution which, though not monomial, still behaves well on toric varieties. This will lead to a proof of the power series formula in Theorem~\ref{thm:degformula} for the dynamical degree.

\subsection{A birational involution}
\label{sec:involution}
The \emph{Cremona involution} on $\bP^d$ is the birational monomial map $\momap_{-I}$ given in affine coordinates by $(y_1,\dots,y_d)\mapsto (y_1^{-1},\dots,y_d^{-1})$, or in homogeneous coordinates $[x_0,\dots,x_d]$, where $y_j=x_j/x_0$, by
\begin{equation*}
  [x_0,\dots,x_d] \to[\prod_{i\ne0}x_i,\dots,\prod_{i\ne d}x_i].
\end{equation*}
It contracts each homogeneous coordinate hyperplane $\{x_j=0\}$ to the torus invariant point where the others intersect and is indeterminate along each linear subspace $\{x_j=x_k=0\}$, $j\ne k$.  

\smallskip
Now consider the $(d+1)\times(d+1)$-matrix $B=(B_{i,j})_{0\le i,j\le d}$ with entries $B_{i,j}=(-1)^{i-j}$ for $i\le j$ and $B_{i,j}=(-1)^{i-j-1}$ for $i>j$. 
It is straightforward to see that $B$ is invertible (except in characteristic two), and that the non-zero entries of the inverse $B^{-1}$ are as follows:
$B^{-1}_{i,i}=\frac12$ for $0\le i\le d$, $B^{-1}_{i,i+1}=\frac12$ for $0\le i<d$, and $B^{-1}_{d,0}=\frac{(-1)^\dim}{2}$.

For example, if $d=3$, then 
\begin{equation*}
  B=
  \begin{bmatrix*}[r]
    1 & -1 & 1 & -1 \\
    1 & 1 & -1 & 1 \\
    -1 & 1 & 1 & -1 \\
    1 & -1 & 1 & 1
  \end{bmatrix*}
  \qquad\text{and}\qquad
  B^{-1}=\frac12
  \begin{bmatrix*}[r]
    1 & 1 & 0 & 0 \\
    0 & 1 & 1 & 0 \\
    0 & 0 & 1 & 1 \\
    -1 & 0 & 0 & 1
  \end{bmatrix*}
  .
\end{equation*}
The matrix $B$ defines an automorphism of $\bP^d$, also denoted by $B$, and given by
\begin{equation*}
  [x_0,\dots,x_d]\mapsto B[x_0,\dots,x_d]=[b_0,b_1,\dots,b_\dim].
\end{equation*}
Now set 
\begin{equation*}
  \inv := B^{-1}\circ \momap_{-I}\circ B.
\end{equation*}
By construction, $\inv\colon\bP^d\tto\bP^d$ is a birational involution, $\crit(\inv)$ consists of the $d+1$ hyperplanes $\{b_j=0\}$, and $\ind(\inv)$ consists of the linear subspaces $\{b_i=b_j=0\}$, $i\neq j$.
One computes that $g=[g_0,\dots,g_d]$, where
\begin{equation}\label{eq:inv}
  g_j=x_j\prod_{i\ne j,j+1}b_i\
  \text{for $j<d$}
  \quad\text{and}\quad
  g_d=x_d\prod_{i\ne0,d}b_i.
\end{equation}
For example, if $d=3$, then
\begin{equation}
  \inv\colon[x_0,x_1,x_2,x_3]
  \to[x_0b_2b_3,x_1b_0 b_3,x_2 b_0 b_1,x_3 b_1 b_2].
\end{equation}
The coordinate hyperplanes $\{x_j=0\}$ are not contained in $\crit(\inv)$, so it follows from the formula above that $\inv$ restricts to a birational self-map on each of them, a statement that will be generalized in Corollary~\ref{cor:sigma} below.

In what follows we will use the (non-toric) hyperplanes
\begin{equation}
  H_j = \{b_j=0\}\cap \torus.
\end{equation}
Then $\inv(H_j\setminus\ind(g))=q_j$, where $q_j\in\bP^\dim\setminus\torus$ is the point with homogeneous coordinates given by column $j$ of the matrix $B^{-1}$.
Also set $H=\bigcup_{j=0}^dH_j$. 
We will write $\overline{H_j}\subset X$ and $\overline{H}\subset X$ for the Zariski closures of $H_j$ and $H$, respectively, in any toric modification $X\to \bP^d$.

Note that the hyperplanes $\overline{H_j}\subset \bP^d$ omit the $d+1$ torus invariant points

  \begin{equation*}
[1,0,\dots,0],[0,1,0,\dots,0],\dots,[0,\dots,0,1].
\end{equation*}

\begin{lemma}
  \label{lem:sigma1}
  Let $\pi\colon X\to \bP^d$ be any toric modification and let 
  $\inv_X\colon X\tto X$ be the lift of $\inv$. Then $(\torus\setminus H)\cap\ind(\inv_X)=\emptyset$ and $\inv_X(\torus\setminus H)\subset\torus$.
\end{lemma}

\begin{proof}
  It suffices to consider the case $X=\bP^\dim$, and then the statement is clear in view of~\eqref{eq:inv}, since $b_i\ne 0$ on $\torus\setminus H$ for $0\le i\le\dim$.
\end{proof}

\begin{lemma}\label{lem:sigma}
Let $\pi\colon X\to \bP^d$ be any toric modification, and $E\subset X$ a pole such that $\cvec_E$ is contained in the interior of a $d$-dimensional cone of $\Sigma(\bP^d)$. Then:
\begin{itemize}
\item[(i)]
$\overline{H}\cap\torus_E=\emptyset$;
\item[(ii)]
the lift $\inv_X\colon X\tto X$ of $g$ is an isomorphism in a neighborhood of the torus $\torus_E$, and sends $\torus_E$ onto itself.
\end{itemize}
\end{lemma}

\begin{proof}
The assumption on $E$ means that $\pi(E)\in\bP^\dim$ is one of the $d+1$ torus invariant points above. As these points do not lie on the closure of $H$ in $\bP^\dim$, we immediately deduce~(i). Moreover, for any $i=1,\dots,d$, the zeros and poles of the rational function $b_i/b_0$ omit all torus invariant points of $\bP^d$.  So on $X$, the restriction of $b_i/b_0$ to $E$ is a non-zero constant.

Now pick $\rvec_1\in M$ such that $\langle \rvec_1,\cvec_E\rangle=1$, and elements $\rvec_2,\dots,\rvec_\dim\in M$ that generate the lattice $\cvec_E^\perp:=\{\rvec\in M\mid \langle \rvec,\cvec_E\rangle=0\}$. Each $\rvec_j$ defines a monomial $\chi_j$ in $(y_1,\dots,y_\dim)$, and $\chi:=(\chi_1,\dots,\chi_\dim)$ gives a birational map of $X$ to $\mathbb{A}^\dim$ which is an isomorphism in a neighborhood of $\torus_E$ and sends $E$ onto the coordinate hyperplane $\{w_1=0\}$ in $\mathbb{A}^\dim$. Now it follows from~\eqref{eq:inv} that $g^*\chi_j=\chi_j\psi_j$, where $\psi_j$ is a monomial in the rational functions $b_i/b_0$ and hence equal to a non-zero constants on $\torus_E$. Thus $\chi\circ\inv_X\colon X\tto\mathbb{A}^\dim$ is also an isomorphism in a neighborhood of $\torus_E$ and sends $E$ onto the coordinate hyperplane $\{w_1=0\}$. We conclude that $\inv_X=\chi^{-1}\circ\chi\circ\inv_X$ has the desired properties.
\end{proof}

\begin{corollary}
\label{cor:sigma}
Let $\pi\colon X\to \bP^d$ be any toric modification. Then the lift $\inv_X\colon X\tto X$ of $\inv$ restricts to a birational map $\inv_X\colon E\tto E$ on any pole $E\subset X$.
\end{corollary}

\begin{proof}
This can be proved in a way similar to Lemma~\ref{lem:sigma}, but let us give a different proof using valuations.  Let $\mathrm{Val}_{\torus}$ be the set of valuations $\val\colon\bk(\torus)^\times\to\R$ on the function field $\bk(\torus)\simeq\bk(y_1,\dots,y_d)$ of the torus $\torus$ that are trivial on $\bk$. We equip it with the topology of pointwise convergence.  The birational map $\inv\colon\torus\tto\torus$ induces a field automorphism of $\bk(\torus)$, and a homeomorphism $g_*$ of $\mathrm{Val}_{\torus}$.
 
We can identify the space $N_\R\cong\R^\dim$ with the set of monomial valuations in coordinates $(y_1,\dots,y_d)$: given $t\in\R^\dim$, the corresponding valuation $\val_t\in\mathrm{Val}_{\torus}$ is uniquely determined by $\val(\sum_\alpha c_\alpha y^\alpha)=\min\{\langle\alpha,t\rangle \mid c_\alpha\ne0\}$ for every Laurent polynomial $\sum_\alpha c_\alpha y^\alpha\in\bk[y_1^{\pm1},\dots,y_\dim^{\pm1}]$. The map $N_\R\to\mathrm{Val}_{\torus}$ is then a homeomorphism onto a closed subset. It is also homogeneous with respect to the multiplicative actions of $\R_{>0}$ on $N_\R$ and $\mathrm{Val}_{\torus}$.

If $X$ is a toric variety and $E\subset X$ a pole, then the valuation corresponding to the element $\cvec_E\in N\subset N_\R$ is also denoted by $\cvec_E$ and can be geometrically described as follows: for any non-zero rational function $f\in\bk(\torus)=\bk(X)$, $\cvec_E(f)\in\Z$ is the order of vanishing of $f$ along $E$.  

It now follows from Lemma~\ref{lem:sigma} that $g_*(\cvec)=\cvec$ for all primitive elements $\cvec\in N$ that lie in the interior of a $\dim$-dimensional cone of $\Sigma_{\bP^\dim}$. Since $g_*$ is homogeneous with respect to the scaling action above, it follows that $g_*(\cvec)=\cvec$ for all $\cvec\in N_\Q$ that lie in the interior of a $d$-dimensional cone of $\Sigma_{\bP^d}$. As the set of such $\cvec$ is dense in $N_\R$ we must have $g_*=\mathrm{id}$ on $N_\R$.

In particular, if $X$ is a toric variety and $E\subset X$ is a pole, then $g_*(\cvec_E)=\cvec_E$. Unraveling the geometric description of $\cvec_E$, this implies that $g_X(E)=E$.
\end{proof}

We now study the critical set of lifts of $\inv$. Set 
\begin{equation}\label{eq:cvecs}
\cvecs := \{\cvec_0,\cvec_1,\dots,\cvec_d\}
\end{equation}
where $\cvec_j\in N\simeq\Z^d$ is the vector whose $k$th entry is the order of $b_j$ in the expression for $g_k/g_0$, $1\le k\le d$.
For example, if $d=3$ we have 
\begin{equation*}
\cvecs = \{(1,1,0),(0,1,1),(-1,-1,0),(0,-1,-1)\}.
\end{equation*}
If $\pi\colon X\to \bP^\dim$ is a toric modification that realizes $\cvec_j\in\cvecs$, then we denote the associated pole by $E_j$.

\begin{proposition}
\label{prop:poleback}
Let $\pi\colon X\to\bP^d$ be any toric modification that realizes all
elements of $\cvecs$, and let $\inv_X\colon X\tto X$ be the lift of of
$\inv$ to $X$. Then:
\begin{enumerate}
\item the irreducible hypersurfaces contracted by $\inv_X$ are $\overline{H_j}$, $0\le j\le \dim$; moreover, $g$ maps a general point on $H_j$ into $\torus_{E_j}$;
\item $\inv_X^* E_j = E_j + H_j$ for each $0\leq j\leq \dim$;
\item $\inv_X^* E = E$ for all other poles of $X$.
\end{enumerate}
\end{proposition}

\begin{proof}
Let $e_j\in\bP^\dim$ denote the point with homogeneous coordinates
equal to the $j$th standard basis vector.  Let $\ell_0$ denote the
line joining $e_0$ and $e_\dim$, and for $0 < j \leq \dim$ let
$\ell_j$ denote the line joining $e_j$ and $e_{j-1}$.  Then $q_j =
g(H_j)$ is a general point on $\ell_j$ and the strict transform of
$\ell_j$ under the toric modification $X\to\bP^\dim$ is the pole
$E_j$.  Hence the preimage of $q_j$ in $X$ is the closure of a
$(\dim-2)$-dimensional subvariety $S_j\subset \torus_{E_j}$.
  
To prove~(i), recall from Corollary~\ref{lem:sigma} that $\inv_X(E) =
E$ for all poles $E\subset X$.  Hence any irreducible hypersurface of
$X$ contracted by $\inv_X$ must meet $\torus\subset X$. It must then
also be contracted by $\inv$, and thus equal to $\overline{H_j}$ for
some $j$. Note, conversely, that $g_X$ contracts each $\overline{H_j}$
to $\overline{S_j}$. Thus~(i) holds.

Let
$\pi\colon Y \to X$ be the (non-toric) blowup of $X$ along each of the
mutually disjoint subvarieties $\overline{S_j}$, and let $\tilde
S_j\subset Y$ denote the preimage of $\overline{S_j}$.
Further, let $Z\to \bP^d$ be the smooth (non-toric) variety
obtained by blowing up all points $q_j = \inv(H_j)$, and $F_j\subset
Z$ the preimage of $q_j$.  Since $\inv$ is
linearly conjugate to the Cremona involution $h_{-I}$, we have that
the lift $\inv_Z\colon Z\tto Z$ of $g$ to $Z$ contracts no
hypersurfaces (i.e. $\inv_Z$ is a `pseudoautomorphism') and exchanges $H_j$ with $F_j$. 

On the other hand, the birational map $\omega\colon Y \tto Z$ induced
by the identity on $\bP^\dim$ satisfies $\omega(\tilde S_j) = F_j$.
Therefore, the irreducible 
hypersurfaces of $Y$ contracted by $\omega$ are precisely the poles
contracted by the toric modification $X\to\bP^\dim$, and in the reverse
direction $\omega^{-1}$ contracts no hypersurfaces of $Z$ at all.  It
follows from this discussion and Corollary \ref{cor:sigma} that the
lift $g_Y\colon Y\tto Y$ of $\inv$ to $Y$ is again a pseudoautomorphism, this time exchanging $\overline{H_j}$ and $\tilde S_j$ while preserving the proper transform $\tilde E$ of each pole $E\subset X$.  

Since the birational map $\inv_{XY} := \inv_X\circ\pi^{-1}$ contracts
no hypersurfaces of $X$, we obtain that $\inv_X^* D = \inv_{XY}^*\pi^*
D$ for all divisors $D$ on $X$.  In particular, for each $0\leq
j\leq\dim$, we have $\inv_X^* E_j = \inv_{XY}^* (\tilde E_j + \tilde
S_j) = \pi_*(\tilde E_j + \overline{H_j}) = E_j + \overline{H_j}$.
And for any other pole $E\subset X$, we have $\inv_X^* E =
\inv_{XY}^*\tilde E = \pi_*\tilde E = E$. Thus~(ii) and~(iii) hold, which completes the proof.
\end{proof}

Now consider the measure 
\begin{equation}
\mu_{\cvecs} := \sum_{j=0}^d \delta_{\cvec_j}
\end{equation}
on $N$. It is balanced in the sense of Corollary~\ref{cor:balanced}.

\begin{proposition}
\label{prop:sigpush}
Let $\mathsf{C}\subset \bP^d$ be an internal curve that meets each critical hyperplane $\overline{H_j}\subset \bP^d$ only at points in $\torus$.  If for some toric modification $X\to \bP^3$ adapted to $\mathsf{C}$ and realizing all elements of $\cvecs$, the proper transform of $\mathsf{C}$ in $X$ avoids the indeterminacy set of $\inv_X$, then $\inv(\mathsf{C})$ is an internal curve satisfying
$$
\mu_{\inv(\mathsf{C})} = \mu_{\mathsf{C}} + (\deg\mathsf{C})\mu_{\cvecs}.
$$
\end{proposition}

\begin{proof}
Recall our convention that internal curves are irreducible.  Since $\mathsf{C}$ must meet some pole of $X$, the assumption $\mathsf{C}\cap \overline{H_j} \subset \torus$ implies that $\mathsf{C}\cap \crit(\inv_X)$ is finite.  It follows that $\inv_X(\mathsf{C})$ is also an internal curve. By Proposition~\ref{prop:pushpull},
$$
(\inv_X({\mathsf{C}}) \cdot E) = ({\mathsf{C}}\cdot \inv_X^* E) 
$$
for every pole $E\subset X$.  When $\cvec_E\notin \cvecs$, this gives $(\inv_X({\mathsf{C}})\cdot E) = ({\mathsf{C}}\cdot E)$.  When $E=E_j$ is the pole associated to $\cvec_j\in\cvecs$, we obtain 
$$
(\inv_X({\mathsf{C}})\cdot E_j) = ({\mathsf{C}}\cdot E_j) + ({\mathsf{C}}\cdot \overline{H_j}).
$$
One should note here that in the term $({\mathsf{C}}\cdot \overline{H_j})$, the closure takes place in $X$.  However, our hypothesis that \emph{in $\bP^d$} all points of ${\mathsf{C}}\cap \overline{H_j}$ lie in $\torus$, means that $({\mathsf{C}}\cdot \overline{H_j}) = \deg {\mathsf{C}}$ is the same if the closure/intersection takes place in $\bP^d$.  The formula for $\mu_{\inv({\mathsf{C}})}$ follows.
\end{proof}

\subsection{The composed birational map}
\label{sec:themap}
We now consider the birational map
$$
f:=\inv\circ\momap\colon\bP^d\tto\bP^d,
$$ 
where $\momap=\momap_A$ is a monomial birational map and $\inv$ is the birational involution we have just discussed. Our aim is to give a power series equation satisfied by  the dynamical degree of $f$ under suitable assumptions on $A$.

Recall the finite subsets $\cP,\cvecs\subset N$ and $\rvecs\subset M$ defined in~\eqref{eqn:p3fan},\eqref{eq:cvecs},  and~\eqref{eq:rvecs}, respectively.
As in the introduction, define $\Psi=\Psi_{\rvecs,\cvecs}\colon\operatorname{Mat}_\dim(\Z)\to\Z_{\ge0}$ by 
\begin{equation}\label{eq:Psi}
\Psi(A)=\sum_{\cvec\in\cvecs}\max_{\rvec\in\rvecs}\langle \rvec,A \cvec\rangle.
\end{equation}

\begin{theorem}
\label{thm:degformula2}
Suppose that $A \in\operatorname{GL}_d(\Z)$ has the property that for all $n\geq 1$, 
each vector in $A^n(\cvecs\cup\cP)$ lies in the interior of a $d$-dimensional cone of $\Sigma(\bP^d)$.  Then $\ddeg = \ddeg(f)$ is the unique positive real number satisfying 
$$
1 = \sum_{n=1}^\infty \Psi(A^n){\ddeg^n}.
$$
\end{theorem}

In~\S\ref{sec:involution} we will see how to find matrices $A$ that satisfy the hypothesis of the theorem.
To prove the theorem, we will analyze the internal curves $f^n(\mathsf{L})$ and $\momap(f^n(\mathsf{L}))$ for a general line $\mathsf{L}\subset\bP^\dim$.
Consider a sequence
\begin{equation*}
\dots\to X_n\to X_{n-1}\to\dots \to X_0\to \bP^d
\end{equation*}
of toric modifications chosen so that $X_0$ realizes all elements of $\cvecs$, and $X_{n+1}$ (further) realizes $A\cvec_E$ for each pole $E\subset X_n$.

For $n\ge1$, let $\inv_n\colon X_n\tto X_n$ and $\momap_n\colon X_{n-1}\tto X_n$ be the lifts of $\inv$ and $\momap$, respectively. Then $f_n:=\inv_n\circ h_n\colon X_{n-1}\tto X_n$ is the lift of $f$ and 
\begin{equation*}
  F_n:=f_n\circ\dots\circ f_1\colon X_0\tto X_n
\end{equation*}
is the lift of $f^n$.
 By convention, $F_0=\mathrm{id}\colon X_0\tto X_0$.
 We also define 
 $F'_n\colon X_0\tto X_n$ 
 for $n\ge 0$ by $F'_0=\mathrm{id}$ and $F'_n:=h_n\circ F_{n-1}$ for $n\ge 1$. Thus $F_n=g_n\circ F'_n$ for $n\ge1$.

\begin{lemma}
\label{lem:nicelines}
Given $n\ge0$, the following hold for a general line $\mathsf{L}\subset\bP^d$, where $\mathsf{C}_0\subset X_0$ is the proper transform of $\mathsf{L}$:
\begin{itemize}
\item[(i)] $\mathsf{C}_0\cap\ind(F_n)=\mathsf{C}_0\cap\ind(F'_n)=\emptyset$;
\item[(ii)] $\mathsf{C}_n:=F_n(\mathsf{C}_0)\subset X_n$ and $\mathsf{C}'_n:=F'_n(\mathsf{C}_0)\subset X_n$ are internal curves;
\item[(iii)] $X_n$ is adapted to $\mathsf{C}_n$ and $\mathsf{C}'_n$;
\item[(iv)]
  for each pole $E\subset X_n$, the intersection $\mathsf{C}_n\cap E$ (resp.\ $\mathsf{C}'_n\cap E$) is empty unless $\cvec_E\in A^n\cP$ or $\cvec_E \in A^k\cvecs$ for some $0\leq k<n$ (resp.\ $0<k<n$);
\item[(v)] if $n\ge 1$, then $\mathsf{C}'_n\cap\overline{H} \subset \torus$  and $\mathsf{C}'_n\cap\ind(\inv_n)=\emptyset$.
\end{itemize}
\end{lemma}

The proof will be given in the next subsection.  Note that the set of lines for which the assertions hold depends on $n$.  When $\bk$ is uncountable, the assertions will hold for \emph{all} $n$ for a very general line,  but we will not need this fact.

\begin{corollary}
\label{cor:fpush}
Fix $n\in\Z_{\geq 0}$.  Then, for a general line $\mathsf{L}\subset\bP^d$, we have
\begin{align}
\mu_{f^n(\mathsf{L})}
&= A_*^n\mu_{\mathsf{L}} + \sum_{j=0}^{n-1}\deg h(f^j(\mathsf{L}))\, A_*^{n-j}\mu_\cvecs\label{eq:meas1}\\
\mu_{h(f^n(\mathsf{L})}
&= A_*^{n+1}\mu_{\mathsf{L}} + \sum_{j=0}^{n-1}\deg h(f^j(\mathsf{L}))\, A_*^{n+1-j}\mu_\cvecs.\label{eq:meas2}
\end{align}
\end{corollary}

\begin{proof}  Equation~\ref{eq:meas1} for $n=0$ is trivial.
A general line has empty intersection in $\torus$ with $\ind(F_n)$ and $\ind(f^n)$, so $f^n(\mathsf{L})$ is an internal curve and equal to the image of $\mathsf{C}_n$ under $X_n\to\bP^\dim$.  Similarly, 
$\momap(f^n(\mathsf{L}))$ is an internal curve equal to the image of $\mathsf{C}'_{n+1}$ under $X_{n+1}\to\bP^\dim$.  The curves $\mathsf{C}_n$ and $\mathsf{C}'_{n+1}$ are described by Lemma~\ref{lem:nicelines}.
  
Corollary~\ref{cor:mopush} first gives
\begin{equation*}
  \mu_{h(f^n(\mathsf{L}))}=A_*\mu_{f^n(\mathsf{L})}.
\end{equation*}
Hence~\eqref{eq:meas1} implies~\eqref{eq:meas2}.
It only remains to show that~\eqref{eq:meas2} implies~\eqref{eq:meas1} for $n+1$. 
But by Lemma~\ref{lem:nicelines}~(v),  Proposition~\ref{prop:sigpush} applies, yielding
\begin{equation*}
  \mu_{f^{n+1}(\mathsf{L})}
  =\mu_{\inv(\momap(f^n(\mathsf{L}))}
  =\mu_{\momap(f^n(\mathsf{L}))}+\deg(\momap(f^n(\mathsf{L})))\mu_{\cvecs},
\end{equation*}
and we are done.
\end{proof}

If $\mathsf{L}$ is a general line, then the degree of the internal curve 
$\momap(f^n(\mathsf{L}))$ can be computed using Corollary~\ref{cor:degree} and~\eqref{eq:intno}:
$$
\deg (\momap\circ f^n)
= (h(f^n(\mathsf{L}))\cdot\{x_0=0\})
= \int \wt \,\mu_{\momap(f^n(\mathsf{L}))},
$$
where $\wt(v)=\max_{u\in\rvecs}\langle u,v\rangle$ is the support function for the coordinate hyperplane $\{x_0=0\}$.
Note that 
$$
 \int\wt\,A_*^n\mu_\cvecs
=\Psi(A^n),
$$
where $\Psi$ is defined in~\eqref{eq:Psi}.
It now follows from integrating $\wt$ against~\eqref{eq:meas2} that 
\begin{equation}\label{eq:rec}
\deg(\momap\circ  f^n) = \deg\momap^n + \sum_{j=0}^{n-1}\Psi(A^{n-j}) \deg(\momap\circ  f^j)
\end{equation}
for all $n\ge 1$; here we have used~\eqref{eq:degmomap}.

Note that because we are using $\mu_\cvecs$ rather than $\mu_\cP$, the integer $\Psi(A^n)$ is \emph{not} necessarily equal to $\deg(\momap_A^n)$.  However, since $\mu_\cvecs$ is balanced and $\cvecs$ spans $N$, the proof of Corollary~\ref{cor:mopush} gives the following result.  It says, in essence, that since the divisor $D$ encoded by $\psi$ is ample, the pullbacks $h^{n*} D$ grow like a bounded multiple of $\deg h^n$.

\begin{lemma} 
\label{lem:dn}
There exists $r\geq 1$ such that 
$r^{-1}\deg\momap_A^n \leq\Psi(A^n) \leq r\deg\momap_A^n$ all $n\in\Z_{\geq 0}$.  In particular, $\lim_{n\to\infty}\Psi(A^n)^{1/n}= \ddeg(\momap_A)$ is the spectral radius of $A$.
\end{lemma}

\medskip
\begin{proof}[Proof of Theorem~\ref{thm:degformula2}]
Let $a(t) := \sum_{n=0}^\infty (\deg\momap\circ f^n) t^n$, $b(t) := \sum_{n=0}^\infty (\deg \momap^n)t^n$, and $c(t) := \sum_{n=1}^\infty\Psi(A^n)t^n$.  Then the recursion formula~\eqref{eq:rec} can be reformulated as a functional equation
$$
a(t)= b(t)+a(t)c(t).
$$ 
For any $n$, we have
$$
(\deg\inv)^{-1}\deg f^{n+1}\le\deg(\momap\circ f^n)\le\deg\momap\deg f^n,
$$
which implies that the radius of convergence of $a(t)$ equals $\ddeg(f)^{-1}$.  Moreover, submultiplicativity of $\deg(f^n)$ implies that $\deg(f^n) \geq \ddeg(f)^n$ for all $n\in\Z_{\geq 0}$.
Hence $a(t)$ strictly increases from $1$ to $\infty$ as $t$ increases from $0$ to $\ddeg(f)^{-1}$.  Similarly $b(t)$ strictly increases from $1$ to $\infty$ as $t$ increases from $0$ to $\ddeg(\momap)^{-1}$, and by Lemma~\ref{lem:dn}, $c(t)$ increases from $0$ to $\infty$ on the same interval.  Hence there is exactly one positive number $t \in (0,\ddeg(\momap)^{-1})$ for which $c(t) = 1$, and from $a = \frac{b}{1-c}$, we conclude that $t=\ddeg(f)^{-1}$ is the radius of convergence of $a(t)$.  
\end{proof}

\subsection{Proof of Lemma~\ref{lem:nicelines}}
As in Theorem \ref{thm:degformula2}, we continue to assume for all $n\geq 1$ that each vector in $A^n(\cvecs\cup\cP)$ lies in the interior of a $d$-dimensional cone of $\Sigma(\bP^d)$.  We start with the following result.

\begin{lemma}
\label{lem:forwardorbit}
Suppose $1\le k\le n$ and that $E\subset X_{k-1}$ is a pole with $\cvec_E\in\bigcup_{j\ge0}A^j(\cvecs\cup\cP)$. Then $\torus_E\cap\ind(f_k)=\emptyset$. Moreover, there exists a pole $E'$ of $X_k$ such that $\cvec_{E'}=A\cvec_E$ and $f_k$ maps $\torus_E$ onto $\torus_{E'}$.
\end{lemma}

\begin{proof}
As above, we write $f_k=\inv_k\circ\momap_k$.
By construction, $A\cvec_E$ is realized as a pole $E'\subset X_k$, so by Proposition~\ref{prop:mo}, $\torus_E$ does not intersect $\ind(h_k)$, and $h_k$ maps $\torus_E$ onto $\torus_{E'}$. Now $\cvec_{E'}\in\bigcup_{j\ge1}A^j(\cvecs\cup\cP)$, so by our assumption on $A$, $\cvec_{E'}$ lies in the interior of a $d$-dimensional cone of $\Sigma(\bP^d)$. Lemma~\ref{lem:sigma} therefore shows that $\torus_{E'}$ does not intersect $\ind(\inv_j)$, and that $\inv_j$ maps $\torus_{E'}$ onto itself. The result follows.
\end{proof}

\begin{proof}[Proof of Lemma~\ref{lem:nicelines}]
Note that a line $\mathsf{L}\subset\bP^\dim$ is internal and adapted to $\bP^\dim$ iff it meets $\torus$ but does not meet the intersection of two distinct coordinate hyperplanes. In this case, it meets each of the coordinate hyperplanes exactly once, transversely, in the corresponding torus. For such lines (and hence for a general line) $X_0$ is adapted to $\mathsf{C}_0=\mathsf{C}'_0$, and~(i)--(iv) hold when $n=0$.
    
Now suppose $n\ge1$. We shall identify a Zariski closed subset $Z_{n,0}\subset\torus$ of codimension at least two such that if an internal line $\mathsf{L}\subset\bP^\dim$ has the property that $\mathsf{L}\cap Z_{n,0}=\emptyset$ and $\bP^\dim$ is adapted to $\mathsf{L}$, then properties~(i)--(v) hold.

By Proposition~\ref{prop:poleback} we can find a Zariski closed subset $Z_n\subset\torus$ of codimension at least two such that $H\cap\ind(\inv_k)\subset Z_n$ for $0\le k\le n$ and $\inv_k(H_j\setminus Z_n)\subset\torus_{E_j}$ for $0\le j\le\dim$, $0\le k\le n$.  Using $Z_n$, we construct Zariski closed subsets $Z'_{n,k}\subset\torus$, $0\le k\le n$ and $Z_{n,k}\subset\torus$, $1\le k\le n$ as follows. First set $Z'_{n,n}:=Z_n$. Then successively define
\begin{equation*}
  Z_{n,k}:=\momap_{k+1}^{-1}(Z'_{n,k+1})
  \quad\text{and}\quad
  Z'_{n,k}:=Z_n\cup(\overline{\inv_k^{-1}(Z_{n,k})\cap(\torus\setminus H}))
\end{equation*}
for $0<k<n$, where the Zariski closure is taken in $\torus$. Finally set $Z_{n,0}:=\momap_1^{-1}(Z'_{n,1})$.
These are all subsets of $\torus$ of codimension at least two since $\momap_{k+1}$ is an automorphism of $\torus$ and $\inv_k\colon\torus\setminus H\to\torus$ is an open embedding.

With these definitions, we obtain the following properties:
\begin{itemize}
\item[(a)]
  if $0\le k<n$ and $p\in\torus\setminus Z_{n,k}$, then $\momap_{k+1}(p)\in\torus\setminus Z'_{n,k+1}$;
\item[(b)]
  if $1\le k\le n$ and $p'\in H_j\setminus Z'_{n,k}\subset X_k$, then $\inv_k(p')\in\torus_{E_j}$, $0\le j\le\dim$;
\item[(c)]
  if $1\le k\le n$ and $p'\in\torus\setminus(H\cup Z'_{n,k})\subset X_k$, then $\inv_k(p')\in\torus\setminus Z_{n,k}$.
\end{itemize}

It follows from these properties and from Lemma~\ref{lem:forwardorbit} that if $p\in\torus\setminus Z_{n,0}\subset X_0$, then $p\not\in\ind(F_n)\cup\ind(F'_n)$. Moreover, either $F_n(p)\in\torus$, or $F_n(p)\in\torus_E$ for some pole $E$ of $X_n$ with $\cvec_E\in\bigcup_{j=0}^{n-1}A^j\cvecs$; and 
either $F'_n(p)\in\torus\setminus Z_n$ or $F'_p(p)\in\torus_E$ for some pole $E$ of $X_n$ with $\cvec_E\in\bigcup_{j=1}^{n-1}A^j\cvecs$.

It also follows from Lemma~\ref{lem:forwardorbit} that if $E\subset X_0$ is a pole with $\cvec_E\in\cP$, then $\torus_E\cap\ind(F_n)=\emptyset$, and $F_n$ maps $\torus_E$ onto $\torus_{E'}$, where $E'\subset X_n$ is the unique pole with $\cvec_{E'}=A^n\cvec_E$.

The above description now shows that if $\mathsf{L}\subset\bP^\dim$ is a line such that $\mathsf{L}\cap Z_{n,0}=\emptyset$ and $\bP^d$ is adapted to $\mathsf{L}$, then properties~(i)--(v) of Lemma~\ref{lem:nicelines} hold.
\end{proof}

\subsection{Proof of Theorem~\ref{thm:degformula}}
\begin{proposition}
\label{prop:powerup}
Suppose the characteristic polynomial of $A\in \SL_\dim(\Z)$ is irreducible over $\Q$ and that its largest roots in $\C$ are a conjugate pair $\maxeval,\bmaxeval$ satisfying $\maxeval^n\notin\R$ for $n\in\Z_{\geq 0}$.  If $\cvec\in\Z^\dim$ is non-zero and $W\subset\R^\dim$ is a rational hyperplane, then there exists a positive integer $N$ such that $A^n \cvec \notin W$ for $n\ge N$.
\end{proposition}

\begin{proof}
By hypothesis $W$ is the orthogonal complement of a non-zero vector $\rvec\in\Z^\dim$.
Let $V'\subset \R^\dim$ denote the real $A$-invariant plane corresponding to the pair $\maxeval,\bmaxeval$ and $V''\subset\R^\dim$ denote its $A$-invariant complement.    Since $A$ is an integer matrix with irreducible characteristic polynomial, neither $V'$ nor $V''$ contain non-zero integer vectors. 

Let $a_n =\langle\rvec,A^n \cvec\rangle$.  Then $(a_n)$ is an integer linear recurrence sequence.  Suppose to get a contradiction that $a_n = 0$ for infinitely many $n\in\Z_{\geq 0}$.  The Skolem-Mahler-Lech Theorem (see 2.5 in~\cite{BGT16}) tells us that if this happens, then $a_n$ vanishes along an arithmetic progression; i.e.\ there exists $m$ and $l$ such that $a_{km+l} = 0$ for all $k\in\Z_{\geq 0}$.  So replacing $\cvec$ by $A^l\cvec$ and then $A$ by $A^m$, we may assume that $a_n = \langle\rvec,A^n\cvec\rangle = 0$ for all $n\in\Z_{\geq 0}$.

However, decomposing $\cvec = \cvec' + \cvec''$ into (necessarily non-zero) vectors $\cvec'\in V'$ and $\cvec''\in V''$, we have that $|\maxeval|^{-mn}\norm{A^n \cvec - A^n \cvec'} \to 0$.  Since $a_n = 0$, we infer that $a'_n :=\langle\rvec, A^n \cvec'\rangle$ satisfies $\lim |\maxeval|^{-mn}a'_n = 0$.   Let $P\in \SL_\dim(\C)$ be a linear change of coordinate such that $PAP^{-1}$ is diagonal with entries $\maxeval^m,\bmaxeval^m,\dots$ equal to the eigenvalues of $A$.  Then $P\cvec' = (z,\bar z, 0,\dots,0)$ for some non-zero $z\in\C$ and $P A^n\cvec' = (\maxeval^{mn} z,\bmaxeval^{mn}\bar z,0,\dots,0)$.  Moreover, the intersection $V'\cap W$ is one-dimensional, hence equal to $\R w$, where $w\in W$ satisfies $Pw = (\omega,\bar\omega,0,\dots,0)$ for some non-zero $\omega\in\C$.  Convergence $|\maxeval^{-mn}|a'_n \to 0$ translates to the statement that $\left(\frac{\maxeval}{|\maxeval|}\right)^{mn} z$ is asymptotic to the line $\R \omega$ as $n\to \infty$.  But this is impossible, because the hypothesis on $\maxeval$ implies that 
$\left\{\left(\frac{\maxeval}{|\maxeval|}\right)^{mn}\in \C:n\in\Z_{\geq 0}\right\}$ is dense in the unit circle.
\end{proof}

\proofof{Theorem~\ref{thm:degformula}}  The complement of the open $\dim$-dimensional cones in $\Sigma(\bP^\dim)$ is contained in the finite union of rational hyperplanes spanned by distinct pairs of vectors in $\rvecs$.  So if, as in the statement of the theorem, $\tilde A\in\SL_\dim(\Z)$ has irreducible characteristic polynomial and leading eigenvalues $\maxeval,\bmaxeval$ with $\maxeval^n\notin \R$ for any $n\in\Z_{\geq 0}$, we can apply Proposition \ref{prop:powerup} to obtain $N\in\Z_{\geq 0}$ such that $\tilde A^n\cvec$ avoids all $(\dim-1)$-dimensional cones of $\Sigma(\bP^\dim)$ for all $n\geq N$.  We can therefore invoke Theorem \ref{thm:degformula2} for $A = \tilde A^N$ to complete the proof.
\qed

\section{Background from diophantine approximation}
\label{sec:sunit}
In this section, we recall fundamental results in Diophantine approximation and basic height bounds that we will use to quantify approximations of the power series in Theorem~\ref{thm:mainb}.

Let $K\subset\bbq$ be a number field, and denote by $M_K$ the places of $K$, with finite places $M_K^{\mathrm{fin}}$ and infinite places $M_K^{\mathrm{inf}}$.  We normalize the absolute values corresponding to elements of $M_K$ so that they extend the absolute values on $\mathbb{Q}$ and satisfy the product formula 
$$
\prod_{\place \in M_K} |z|_\place = 1
$$
for all $z \in K \setminus \{ 0 \}$.
In particular, if $\place\in M_K^{\mathrm{inf}}$ is a complex place, corresponding to a conjugate pair $\tau,\bar{\tau}\colon K\to\C$ of complex embeddings, then $|x|_v = |\tau(x)|^{2} = |\bar{\tau}(x)|^{2}$.

We make use of a result of Evertse on linear forms.  Given a finite set of places $S \subset M_K$ that contains all the infinite places, we let $\cO_K(S)$ denote the set of $S$-integers in $K$; that is, 
$$
\cO_K(S) := \{ z \in K : |z|_\place \leq 1 \ \forall \place \not\in S \}.
$$  
Then $\cO_K(S)$ is a ring, which is called the ring of $S$-\emph{integers} in $K$, and the units of $\cO_K(S)$ are call the $S$-\emph{units}.  In the case when $S$ is the set of infinite places of $K$, we write $\cO_K$ for $\cO_K(S)$, which is the ring of algebraic integers in $K$. 

Given a vector ${\bf z} = (z_1, \dots, z_{\ell}) \in K^{\ell}$ we set
$$
H_S({\bf z}) := \prod_{\place \in S} \max \{ |z_1|_\place, \dots, |z_{\ell}|_\place \}.
$$

We use the general result of Evertse \cite{Eve84} on unit equations, as formulated in \cite{EG}.

\begin{theorem}[\cite{EG}, Proposition 6.2.1] \label{thm:Evertse} Let $S \subset M_K$ be a finite set of places of $K$ containing all infinite places, $T$ a subset of $S$, and $\ell \geq 2$ an integer. For any fixed $\epsilon > 0$, there exists a constant $c=c(K, S, \ell, \epsilon)$ so that if ${\bf z} = (z_1, \dots, z_{\ell}) \in \cO_K(S)^{\ell}$ and $\sum_{k \in I} z_k \ne 0$ for all non-empty $I \subset \{ 1, \dots, \ell \},$ then 
$$
\prod_{\place \in T} |z_1 + \cdots + z_{\ell} |_\place \geq c \frac{\prod_{\place \in T} \max \{ |z_1|_\place, \dots, |z_{\ell}|_\place \}}{H_S({\bf z})^{\epsilon} \prod_{\place \in S} \prod_{k = 1}^{\ell} |z_k|_\place}.
$$
\end{theorem}

When comparing non-negative sequences we will use Vinogradov notation `$a_j \gg b_j$' to mean that $b_j \leq C a_j$ for some constant $C > 0$ and large enough $j\in\Z_{\geq 0}$.

\begin{corollary}
\label{cor:xiandbarxi} Let $K\hookrightarrow\C$ be a number field together with an embedding into $\C$. 
Let $\xi\in K$ be such that $\xi^j\notin\R$ whenever $j$ is a non-zero integer.  Then for each non-zero $a\in K$, $j\in\Z_{\geq 0}$, and any $r\in (0,1)$, we have
$$
|\re a\xi^j| \gg |r\xi|^j
$$
\end{corollary}

We emphasize that for our purposes, it is important that $r$ can be taken arbitrarily close to $1$ in this corollary.

\begin{proof}
By hypothesis $\xi = |\xi|e^{2\pi it}$ where $t\in [0,1)$ is irrational.  
Let $\mathbf{z} = (z_1,z_2) = (a\xi^j,\bar a \bar\xi^j)$.  Since $t \in\R\setminus\Q$, we see that $z_1+z_2 = 0$ for at most one $j\in\Z_{\geq 0}$.
Take $S\subset M_K$ to be a finite set containing all infinite places of $K$ so that $a, \bar a,\xi,\bar\xi$ are $S$-units.  Set $T=\{|\cdot|^2\}$.  Then the product formula tells us that $\prod_{\place\in S} |z_i|_\place = 1$ for $i=1,2$.  Hence for every $\epsilon>0$, Theorem~\ref{thm:Evertse} gives
$$
|2\re a\xi^j|^2=|a\xi^j + \bar a\xi^j|^2 \gg |\xi|^{2j} H_S(\mathbf{z})^{-\epsilon}.
$$
As 
$$
H_S(\mathbf{z}) \ll \left( \prod_{\place \in S} \max \{ |\xi|_\place, |\bar\xi|_\place \} \right)^j \leq \left( \prod_{\place \in S} \max \{ |\xi|_\place, |\bar\xi|_\place, 1 \} \right)^j,
$$
the conclusion follows by choosing $\epsilon> 0$ so that 
$$
\left( \prod_{\place \in S} \max \{ |\xi|_\place, |\bar\xi|_\place, 1 \} \right)^{-\epsilon} = r^2,
$$
noting that because $\xi$ is not a root of unity, the product in parentheses is strictly greater than $1$.
\end{proof}

For further applications of Theorem~\ref{thm:Evertse}, we need a basic height bound.  Here we employ multi-index notation 
\begin{equation}
\label{eqn:multiindex}
\rho^\alpha := \rho_1^{\alpha_1}\dots \rho_\dim^{\alpha_\dim}\qquad {\rm and}\qquad {\rm deg}(\alpha) = \sum \alpha_i 
\end{equation}
for all $\alpha = (\alpha_1,\dots,\alpha_\dim)\in \Z^\dim$.

\begin{lemma} \label{lem:heightbound} Let $D$ be a positive integer and let $Z$ be a finite subset of $K$.  Then there exists a positive constant $R$ such that whenever $z = \sum_{{\rm deg}(\alpha) \leq D}\zeta_\alpha \rho^\alpha$ 
is a polynomial of degree at most $D$ with coefficients $\zeta_\alpha \in Z$, we have
$$
\prod_{\place \in M_K} \max \{ |z|_\place, 1 \} \ll R^D.
$$
\end{lemma}

\begin{proof} 
See the proof of Lemma 3.6 in \cite{BDJ20}.
\end{proof}

Another result, useful for establishing non-degeneracy in Theorem~\ref{thm:Evertse}, is the following unit equations theorem of Evertse, Schlickewei, and Schmidt.  We recall that a finite sum $\sum_{i\in I} a_i z_i$ is \emph{non-degenerate} if no proper subsum vanishes; that is, $\sum_{j\in J} a_j z_j \neq 0$ for all $J \subsetneq I$. We say that an abelian group $G$ is of \emph{finite rank} if there exists a finitely generated subgroup $G'$ of $G$ such that every element of $G/G'$ has finite order.

\begin{theorem}[\cite{ESS02}] \label{thm:unitequations} Let $G \subset \C^*$ be a multiplicative subgroup of finite rank and let $\ell$ be a positive integer.  Then for each $(a_1,\dots,a_\ell)\subset \C^{\ell}$, there are only finitely many non-degenerate sums
$$
\sum_{i\in I} a_i z_i = 1, 
$$
where $I\subset \{1,\dots,\ell\}$ and $z_i \in G$ for all $i\in I$.
\end{theorem}

\section{Proof of Theorem \ref{thm:main}}
\label{sec:prelim}
We now explain how Theorem~\ref{thm:main} may be reduced to proving Theorem~\ref{thm:mainb}.  The bulk of the work will be to establish an auxiliary result that gives us the discordance condition needed to employ Theorem~\ref{thm:mainb}.  

\subsection{Setup}
Recall the relevant notation and assumptions from Theorem~\ref{thm:main}: $\tilde A\in\SL_\dim(\Z)$
is a matrix of size $\dim\geq 3$ with irreducible characteristic polynomial, and $\rvecs,\cvecs \subset\Z^\dim$ are finite sets of vectors with $\#\rvecs \geq 2$.  In particular, the set 
$$
\dvecs:=(\rvecs-\rvecs)\setminus\{0\}
$$ 
is non-empty.  By hypothesis there are no angular resonances among the eigenvalues $\evals = (\eval_1,\dots,\eval_\dim)$ of $\tilde A$, and the eigenvalues of largest modulus are a complex conjugate pair $\eval_1=\maxeval,\eval_2 = \bmaxeval$. In particular, 
\begin{equation*}
  \theta:=\tfrac1{2\pi}\arg(\maxeval)\in(0,1)
\end{equation*}
is an irrational number.

We can extend the function $\fnal:=\fnal_{\rvecs,\cvecs}\colon\mats_\dim(\Z) \to \Z$ from Theorem~\ref{thm:main} to all of $\mats_\dim(\C)$ by
\begin{equation}
\label{eq:fnal}
\fnal(A) := \sum_{\cvec\in\cvecs}\max_{\rvec\in\rvecs}\re\langle\rvec,A\cvec\rangle = \sum_{\cvec\in\cvecs}\re\langle\Gamma(A\cvec),A\cvec\rangle, 
\end{equation}
where $\langle z,z'\rangle=\sum_{i=1}^\dim z_iz'_i$ is the $\C$-bilinear pairing on $\C^\dim$, and for each $z\in\C^\dim$, the vector $\Gamma(z)\in\rvecs$ is chosen so that 
$$
\re\langle\Gamma(z),\cvec\rangle = \max_{\rvec\in\rvecs} \re\langle\rvec,z\rangle.
$$
Hence $\Gamma:\C^\dim\to\rvecs$ is uniquely determined and locally constant outside the finite collection of real hyperplanes given by
\begin{equation}
\label{eq:hyper}
\bigcup_{\dvec\in\dvecs} \{\cvec\in \C^\dim \colon  \re\langle\dvec,\cvec\rangle=0\},
\end{equation} 
where as above $\dvecs$ consists of differences between distinct elements of $\rvecs$.

Now let $K\subset\bbq$ be a splitting field for the characteristic polynomial of $\tilde A$.  Then the $\tilde A$-equivariant projection $\tilde\pi\colon\C^\dim\to\C^\dim$ onto the $\maxeval$-eigenspace of $\tilde A$ is defined over $K$.  For any $Y\in\SL_\dim(\Z)$, we set 
$$
A_Y := Y^{-1}\tilde AY
$$ 
and let $H_Y\subset\C^\dim$ be the $\maxeval$-eigenspace of $A_Y$.  The $A_Y$-equivariant projection $\pi_Y\colon\C^\dim\to\C^\dim$ onto $H_Y$ is then given by $\pi_Y(\cvec) = Y^{-1}\tilde\pi(Y\cvec)$.

Since $A_Y$ is an integer matrix with irreducible characteristic polynomial, no proper $A_Y$ or $A_Y^T$ invariant subspace of $\R^\dim$ contains non-zero integer vectors. Thus, for any non-zero $\dvec\in\Z^\dim$, the linear function $\re\langle\dvec,\cdot\rangle$ does not vanish identically on $H_Y$. Indeed, since $H_Y$ is $A_Y$-invariant, the subspace $H_Y^\perp$ of $\C^\dim$ consisting of vectors $z$ for which $\re\langle z,\cdot\rangle$ vanishes is $A_Y^T$ invariant and proper and therefore omits all integer vectors.
It follows that the restriction $\Gamma|_{H_Y}$ is nonconstant, though still uniquely defined and locally constant outside a finite union of real rays in the complex line $H_Y$.

\subsection{Reducing Theorem~\ref{thm:main} to a discordance condition}
\label{sec:reduction}
To prove Theorem~\ref{thm:main}, we need to study $\fnal(A_Y^j)$ for large $j$, and
Equation~\eqref{eq:fnal} reduces this to understanding $\Gamma(A_Y^j\cvec)$ for $\cvec\in\cvecs$.

\begin{lemma}
\label{lem:restriction}
Let $\cvec\in\Z^\dim$ be a non-zero vector.  Then for all but finitely many $j\in\Z_{\geq 0}$, we have
$$
\Gamma(A_Y^j\cvec) = \Gamma(\maxeval^j \pi_Y(\cvec)),
$$  
and the common value is a vector $\rvec\in\rvecs$ that uniquely maximizes both $\langle\rvec,A_Y^j\cvec\rangle$ and 
$\re\langle\rvec,\maxeval^j\pi_Y(\cvec)\rangle$.
\end{lemma}

\begin{proof}
Since the characteristic polynomial of $\tilde A$ is irreducible over $\mathbb{Q}$ and since $\cvec,\dvec$ are non-zero integer vectors, $\pi_Y(\cvec)$ is non-zero and $\re\langle\dvec,\cdot\rangle$ does not vanish identically on $H_Y$.  So (real) linearity and the fact that $\xi^j\notin\R$ for any $j\in\Z_{>0}$ imply that $\re\langle\dvec,\maxeval^j\pi_Y(\cvec)\rangle=0$ for at most one $j\in\Z_{\geq 0}$.

Hence, for sufficiently large $j$, we have $\re\langle\rvec,\maxeval^j\pi_Y (\cvec)\rangle > \re\langle\tilde\rvec,\maxeval^j\pi_Y(\cvec)\rangle$, where $\rvec = \Gamma(\maxeval^j\pi_Y(\cvec))$ and $\tilde \rvec\in \rvecs\setminus\{\rvec\}$ is any other vector.  

We need to show that $\rvec$ also uniquely maximizes $\langle\rvec,A_Y^j\cvec\rangle$ for large $j$. Now
$$
\cvec' := \cvec - \pi_Y(\cvec) - \overline{\pi_Y(\cvec)}
$$
lies in the $A_Y$-invariant subspace of $\C^\dim$ complementing $H_Y\oplus \bar H_Y$, so since $\maxeval,\bmaxeval$ are the eigenvalues of maximal magnitude, there exists $\epsilon>0$ such that
$$
\norm{A_Y^j\cvec'}\le (1-\epsilon)^j |\maxeval |^{j}
$$
for sufficiently large $j$.
On the other hand, given $\tilde\rvec \in \rvecs\setminus\{\rvec\}$ we can apply Corollary~\ref{cor:xiandbarxi} with $a=2\langle\rvec-\tilde\rvec,\pi_Y(\cvec)\rangle$ and some fixed $r \in (1-\epsilon/2,1)$.  Then from maximality of $\re\langle\rvec,\maxeval^j\pi_Y(\cvec)\rangle$ we obtain for large $j$ that
$$
2\re\langle\rvec-\tilde\rvec,\maxeval^j\pi_Y(\cvec)\rangle\ge(1-\epsilon/2)^j |\maxeval|^{j}.
$$
Hence
\begin{align*}
  \langle\rvec - \tilde\rvec,A_Y^j\cvec\rangle
  &= 2\re \langle\rvec-\tilde\rvec,\maxeval^j\pi_Y(\cvec)\rangle
    +\langle\rvec - \tilde\rvec,A_Y^j\cvec'\rangle \\
  &\ge((1-\epsilon/2)|\maxeval|)^j - ((1-\epsilon)|\maxeval|)^j > 0
\end{align*}
for all $j$ sufficiently large, completing the proof.
\end{proof}

Write $A_Y = P D P^{-1}$, where $D$ is the diagonal matrix with entries $\eval_1,\dots,\eval_\dim$ and $P=P_Y$ is a matrix with $i$th column equal to an eigenvector for $\eval_i$.  By our assumptions, the number field $K$ contains all entries of $P$ and $D$.

For large $j\in\Z_{\geq 0}$, Lemma~\ref{lem:restriction} tells us that
\begin{equation}
\label{eq:Psi2}
\fnal(A_Y^j)
= \sum_{\cvec\in\cvecs}\langle P^T\Gamma(\maxeval^j\pi_Y(\cvec)),D^j P^{-1}\cvec\rangle
= \pair{\gamma_Y(j\theta)}{(\eval_1^j,\ldots ,\eval_m^j)},
\end{equation}
where
\begin{equation}\label{eq:gam1}
\gamma_Y = (\gamma_{Y,1},\dots,\gamma_{Y,\dim})\colon\R\to K^\dim
\end{equation} is the piecewise constant and $1$-periodic function with $i$th component given by
\begin{equation}\label{eq:gam2}
  \gamma_{Y,i}(t) = \sum_{\cvec\in\cvecs}
  (P^T\Gamma(e^{2\pi\imunit t} \pi_Y(\cvec)))_i (P^{-1} \cvec)_i.
\end{equation}

\begin{lemma}\label{lem:maxfcn}
For large $j\ge1$, the 1-periodic function on $\R$ given by
\begin{equation*}
    f_j(t):=\langle\gamma_Y(t),(\eval_1^j,\dots,\eval_\dim^j)\rangle
\end{equation*}
is $\Z$-valued, non-constant, and maximized at $t=j\theta$.  As a consequence, the discontinuity set $\discty(\gamma_Y)\subset\R$ of $\gamma_Y$ is non-empty. 
\end{lemma}

\begin{proof}
Unwinding the definition of $\gamma_Y$, it follows that 
\begin{equation*}
f_j(t)=\sum_{\cvec\in\cvecs}\langle\Gamma(e^{2\pi\imunit t}\pi_Y(\cvec)),A_Y^j\cvec\rangle.
\end{equation*}
For any $\cvec\in\cvecs$ we have $\Gamma(e^{2\pi\imunit t}\pi_Y(\cvec))\in\rvecs\subset\Z^\dim$ and $A_Y^j\cvec\in\Z^\dim$, so $f_j(t)\in\Z$.  Moreover, for any large enough $j$,
\begin{align*}
  \langle\Gamma(e^{2\pi\imunit t}\pi_Y(\cvec)),A_Y^j\cvec\rangle
  &\le\langle\Gamma(A_Y^j\cvec), A_Y^j\cvec\rangle\\
  &=\langle\Gamma(\maxeval^j\pi_Y(\cvec)), A_Y^j\cvec\rangle\\
  &=\langle\Gamma(e^{2\pi\imunit j\theta}\pi_Y(\cvec)), A_Y^j\cvec\rangle,
\end{align*}
where the inequality holds by definition of $\Gamma$, the first equality follows from Lemma~\ref{lem:restriction}, and the second equality follows from homogeneity of $\Gamma$. Thus $f_j(t)$ is maximized for $t=j\theta$. It only remains to show that $f_j$ is non-constant. But if $f_j$ were constant, the inequality above would have to be an equality for all $t$ and all $\cvec\in\cvecs$. By the uniqueness statement in Lemma~\ref{lem:restriction}, this would imply that $\Gamma(e^{2\pi\imunit t}\pi_Y(\cvec))$ is a constant function of $t$. Since $\pi_Y(\cvec)\ne0$, the 0-homogeneous function $\Gamma|_{H_Y}$ would then be constant, a contradiction.
\end{proof}

In order to prove Theorem~\ref{thm:main} we will require the following result, whose proof will be given in the following subsection. 

\begin{theorem}
\label{thm:noresonances}
There exists a coset $\mathcal{Y}\subset \SL_\dim(\Z)$ of a finite-index subgroup such that $\theta$ and $\discty(\gamma_Y)$ are discordant for every $Y\in\mathcal{Y}$.
\end{theorem}

Recall that discordance means that
for any $t,t' \in \discty(\gamma_Y)\cup\{0\}$ and any $a,b\in\Z$, $a\theta = b(t-t')\mymod 1$ implies $a=0$ and either $t-t'\in\Z$ or $b$ is even;

Taking Theorem~\ref{thm:noresonances} for granted momentarily and assuming Theorem~\ref{thm:mainb}, we can quickly give the proof of Theorem~\ref{thm:main}.

\medskip
\proofof{Theorem~\ref{thm:main}}
Suppose $Y\in\SL_\dim(\Z)$ is such that $\theta$ and $\gamma:=\gamma_Y$ is discordant, and fix  $N\ge1$. Then $N\theta$ and $\gamma$ are also discordant.

Since the largest eigenvalues of $\tilde A$ are $\maxeval,\bmaxeval$, the radius of convergence of the series \eqref{eqn:theseries} in Theorem~\ref{thm:main} is at least $|\maxeval|^{-N}$. 
Pick any $x\in \bbq\cap(0,|\maxeval|^{-N})$, and set
\begin{equation}
\rho_i = (x\eval_i)^{N} \qquad {\rm for}~i=1,\ldots,\dim,
\end{equation}
so that $|\rho_i|<1$ for all $i$.

In view of Theorem~\ref{thm:noresonances} and Equation~\eqref{eq:Psi2}, it suffices to show that 
\begin{equation}
\label{eqn:omegaformula}
\Omega := \sum_{j=1}^\infty \pair{\gamma(jN\theta)}{(\rho_1^j,\ldots ,\rho_m^j)}
\end{equation}
is transcendental. Indeed, the series in Theorem~\ref{thm:main} might differ from this one in finitely many terms, but this is immaterial since all terms are algebraic.  

We must show that the hypotheses of Theorem~\ref{thm:mainb} are satisfied. 
First, the assumption that there are no angular resonances between distinct eigenvalues 
$\eval_i$ of $A$ implies that the $\rho_i$ are pairwise multiplicatively independent. Indeed, suppose $\rho_i^a\rho_j^b=1$, where $i\ne j$ and $(a,b)\ne(0,0)$. As $|\rho_i|<1$, we can't have $a,b\ge0$ or $a,b\le0$, so we may assume $a>0$ and $b=-c<0$. Then $\xi_i^{Na}=x^{N(a+c)}\xi_j^{Nc}$, and hence $\xi_i^{Na}\bar\xi_j^{Nc}=x^{N(a+c)}|\xi_j|^{2Nc}>0$, a contradiction.

Second, we have already observed that $N\theta$ and $\gamma$ are discordant. 

It therefore only remains to show that the maximality condition in Theorem~\ref{thm:mainb} holds. But this amounts to, for $j$ large, the 1-periodic function
\begin{equation*}
  t\mapsto \langle\gamma(Nt),(\rho_1^j,\dots,\rho_\dim^j)\rangle
  =x^{Nj}f_j(Nt)
\end{equation*}
being $\R$-valued, non-constant, and maximized at $t=jN\theta$. Here $f_j$ is the function in Lemma~\ref{lem:maxfcn}, which therefore allows us to conclude the proof.
\qed

\subsection{Establishing discordance}
\label{sec:resonances}

We will spend the rest of this section proving Theorem~\ref{thm:noresonances}.  The discontinuities of $\gamma_Y$ all arise from discontinuities of $\Gamma|_{H_Y}$. More precisely,~\eqref{eq:gam2} shows that $t_0\in\discty(\gamma_Y)$ implies that the function $t\mapsto\Gamma(e^{2\pi\imunit t}\pi_Y(\cvec))$ is discontinuous at $t_0$ for some $\cvec\in\cvecs$. This, in turn, means that 
$\langle\dvec,e^{2\pi\imunit t_0}\pi_Y(\cvec)\rangle$ is purely imaginary for some $\dvec \in \dvecs= (\rvecs-\rvecs)\setminus\{0\}$;  hence $e^{4\pi\imunit t_0}$ is one of the finitely many elements of $K$ of the form
\begin{equation}
\label{eqn:candidatedirection}
 \rfn(Y,\cvec,\dvec) := -\frac{\overline{\langle\dvec,\pi_Y(\cvec)\rangle}}{\langle\dvec,\pi_Y(\cvec)\rangle}.
\end{equation}

Most of the time we will fix $\cvec$ and $\dvec$ and regard $\rfn\colon\SL_\dim(\Z)\to K$ as a function of $Y$ only.  
To obtain Theorem~\ref{thm:noresonances} we will show that for ``many'' $Y\in\SL_\dim(\Z)$, $\sigma(Y)$ is not a unit in the ring of algebraic integers of $K$, see Corollary~\ref{cor:notintegers} below.

Fix $\maxeval$-eigenvectors $\cmax,\rmax\in K^\dim$ of $\tilde A$ and $\tilde A^T$, respectively, normalized so that $\langle\rmax,\cmax\rangle = 1$.  Then the projection $\tilde\pi\colon\C^\dim\to\C^\dim$ onto the $\maxeval$-eigenspace of $\tilde A$ is given by $\tilde\pi(\cvec) = \langle\rmax,\cvec\rangle\cmax$.
For $Y\in \SL_\dim(\Z)$ we further have that
$$
\pi_Y(\cvec) = \langle\rmax,Y\cvec\rangle Y^{-1}\cmax.
$$
Hence we can rewrite
\begin{equation}
\label{eq:ru}
\rfn(Y,\cvec,\dvec) = -\frac{\langle\brmax,Y\cvec\rangle\langle\dvec, Y^{-1}\bcmax\rangle}{\langle\rmax,Y\cvec\rangle\langle\dvec,Y^{-1}\cmax\rangle}.
\end{equation}
This formula extends $\rfn$ to a rational function $\rfn\colon\mats_\dim(\C)\dashrightarrow\C$ on the space $\mats_\dim(\C)\simeq\C^{\dim^2}$, with the homogeneity property $\rfn(zY) = \rfn(Y)$ for all $z\in\C^*$.

Note that $\sigma$ is regular and non-zero at any $Y\in\SL_\dim(\Z)$, since neither the numerator nor the denominator of~\eqref{eq:ru} can vanish. For example $\langle\rmax,Y\cvec\rangle\ne0$ since $Y\cvec\in\Z^\dim$ is non-zero and the entries of $\rmax$ are $\Q$-linearly independent, given that $\tilde A$ has irreducible characteristic polynomial.

\begin{lemma}
\label{lem:nonconstant}
Given $\cvec,\tilde\cvec\in\cvecs$ and $\dvec,\tilde\dvec\in\dvecs$, define $\rfn,\tilde \rfn\colon\mats_\dim(\C)\to\C$ by
$$
\rfn(Y) := \rfn(Y,\cvec,\dvec)\quad\text{and}\quad \tilde\rfn(Y) = \rfn(Y,\tilde\cvec,\tilde\dvec).
$$
Then 
\begin{itemize}
\item[(i)]
  $\rfn$ and $\tilde\rfn$ are non-constant;
\item[(ii)] either $\rfn/\tilde\rfn$ is non-constant or $\rfn\equiv\tilde\rfn$, the latter occurring precisely when $\cvec$ is a multiple of $\tilde\cvec$ and $\dvec$ is a multiple of $\tilde\dvec$.
\end{itemize}
\end{lemma}

\begin{proof}
Recall that all vectors in $\cvecs$ and $\dvecs$ are non-zero.

We first prove~(i),
supposing to get a contradiction that $\rfn(Y)=\rfn_0$ for some constant $\rfn_0$ and every $Y\in \SL_\dim(\Z)$.  Since $\SL_\dim(\Z)$ is Zariski dense in $\SL_\dim(\C)$, and $\rfn(zY) = \rfn(Y)$ for all $z\in\C^*$ and $Y\in {\rm GL}_\dim(\C)$, we infer that $\rfn(Y) = \rfn_0$ for all $Y\in{\rm GL}_\dim(\C)$.  Taking $H\subset\mats_\dim(\C)$ to be the complex hyperplane of matrices $Y$ such that $\langle\rmax,Y\cvec\rangle = 0$, we note that since the (irreducible) variety $\{\det(Y) = 0\}$ contains no hyperplanes, invertible matrices are Zariski dense in $H$.  So for general $Y\in H$, we have
$$
0= \sigma_0\langle\rmax,Y\cvec\rangle\langle\dvec,Y^{-1}\cmax\rangle = \langle\brmax,Y\cvec\rangle\langle\dvec, Y^{-1}\bcmax\rangle. 
$$
Hence one of the two factors on the right vanishes identically.  But 
$\maxeval\notin \R$ implies that $\rmax$ and $\brmax$ are linearly independent.  So $\langle\brmax,Y\cvec\rangle\neq 0$ outside a proper linear subspace of $H$, and it must be that $\langle\dvec,Y^{-1}\bcmax\rangle= 0$ for every invertible $Y\in H$.  

But this amounts to saying that there is a hyperplane $H'\subset \mats_\dim(\C)$ such that for all invertible $Y\in H$, we have $Y^{-1}\in H'$.  To see that this is impossible, choose matrices $B_1,B_2\in{\rm GL}_\dim(\C)$ such that $B_1^T\rmax = (1,0,\dots,0)$ and $B_2\cvec = (0,\dots,0,1)$.  Replacing all $Y\in H$ by $B_1^{-1}YB_2^{-1}$, we may assume that $H$ is the set of matrices whose $(1,\dim)$-entry is zero.  It follows that $H$ and therefore also $H'$ contains all diagonal matrices.  And for any distinct $i,j\in\{1,\dots,\dim\}$ with $(i,j)\neq (1,\dim)$, we have $Y = I+E_{ij}\in H$, where $E_{ij}$ is the matrix with $(i,j)$-entry equal to $1$ and all other entries equal to zero.  Thus $Y^{-1}=I-E_{ij}\in H'$, and we infer from taking linear combinations that $H\subset H'$.  Finally, $H$ also contains the upper triangular matrix $Y = I + \sum_{i=1}^{\dim-1} E_{i,i+1}$ whose $(\dim,1)$-minor has non-zero determinant.  Hence $Y$ is invertible and by Cramer's formula for $Y^{-1}$, the $(1,\dim)$-entry of $Y^{-1}$ is non-zero.  It follows that $H'$ is strictly larger than $H$ and in particular, not a hyperplane.  

It remains to prove~(ii), so assume instead that $\rfn/\tilde\rfn$ is constant on $\SL_\dim(\Z)$. As before, this identity extends to all of ${\rm GL}_\dim(\C)$.  Also assume that $\cvec$ and $\tilde\cvec$ are not proportional; the case when $\dvec$ and $\tilde\dvec$ are not proportional is similar.

We again let $H$ be the set of $\dim\times\dim$ complex matrices $Y$ for which $\langle\rmax,Y \cvec\rangle = 0$.  This time, we obtain from the formulas for $\rfn$ and $\tilde\rfn$ that 
$$
0 = \langle\brmax,Y\cvec\rangle\langle\dvec,Y^{-1}\bcmax\rangle\langle\rmax,Y\tilde \cvec\rangle\langle\tilde\dvec,Y^{-1}\cmax\rangle
$$
for all $Y\in H$.  Thus one of the four factors on the right vanishes identically.  But we already showed that the first two factors can't vanish, and the fourth factor may be excluded by the same argument used to rule out the second.  Finally, since $\tilde\cvec$ is not a multiple of $\cvec$, we exclude the third factor for the same reason as the first.  So we again have our contradiction.
\end{proof}
 
\begin{lemma}\label{lem:nonconst2}
Let $\rfn\colon\SL_\dim(\C)\tto \C$ be a non-constant rational function that is regular at any element of $Y\in\SL_\dim(\Z)$. Then there exists a matrix $Y\in\SL_\dim(\Z)$ and a nilpotent matrix $B\in\mats_\dim(\Z)$ such that the function $\tau\colon\Z\to\C$ given by $\tau(k) = \rfn(Y(I+kB))$ is non-constant. 
\end{lemma}

\begin{proof}
Suppose that $\tau$ is constant for all choices of $Y$ and $B$. Working inductively, we then have
$$
\rfn((I+k_1B_1)\cdot \dots \cdot (I+k_sB_s)) = \rfn((I+k_1B_1)\cdot \dots \cdot (I+k_{s-1}B_{s-1})) = \dots = \rfn(I)
$$
for all $k_1,\dots,k_s\in \Z$ and nilpotent $B_1,\dots, B_s\in\mats_\dim(\Z)$.  By~\cite{GT} the group generated by unipotent matrices $I+kB$ is a finite-index subgroup of $\SL_\dim(\Z)$ and therefore Zariski dense in $\SL_\dim(\C)$.  It follows that $\rfn(Y) \equiv \rfn(I)$ is constant on $\SL_\dim(\C)$.
\end{proof}
 
Recall that $\cO_K$ denotes the subring of integers in the splitting field $K$ and $\cO_K^*$ denotes its group of units, a finitely generated abelian group.  Recall also (see \S\ref{sec:sunit}) that $M_K^{\mathrm{fin}}$ denotes the set of finite places $\place$ on $K$ and $\abs[\place]{\cdot}$ denotes the associated absolute values.  Every such absolute value extends (up to normalization) the $p$-adic absolute value $\abs[p]{\cdot}$ on $\Z$ associated to the unique prime $p\in\Z_{\geq 0}$ for which $\abs[\place]{p} < 1$.  Conversely, for any prime $p$ there are finitely many $\place\in M_K^{\mathrm{fin}}$ such that $\abs[\place]{p}<1$.  Recall that if $a\in K$, then $a\in\cO_K^*$ iff $|a|_\place=1$ for all $\place\in M_K^{\mathrm{fin}}$.

The following lemma is well-known and can be deduced from a result of Schur~\cite{GB,Schur}. Since we lack a precise reference, we give a different proof.  It depends on two distinct ways to determine whether a sequence $(\tau(k))_{k\geq 0}\subset\C$ satisfies a linear recurrence $\tau(k+n) = \sum_{0\leq j <n} c_j \tau(k+j)$.  First, if $\tau(k)$ is the restriction of a rational function $\tau:\C\to\C$, then $(\tau(k))$ satisfies a linear recurrence if and only if the rational function is a polynomial.  Second, we have the more standard general fact that $(\tau(k))$ satisfies a linear recurrence if and only if its generating function $\sum \tau(k) z^k$ is rational.  

\begin{lemma}
\label{lem:manyprimes}
Let $\tau\in K(x)$ be a non-constant rational function.  Then there are infinitely many places $\place\in M_K^{\mathrm{fin}}$ for which there is some integer $k$ such that $\abs[\place]{\tau(k)} \neq 1$. 
\end{lemma}

\begin{proof} Write $\tau=\alpha/\beta$ as a quotient of coprime polynomials $\alpha,\beta\in K[x]$.  Then the sequences $(\alpha(k))_{k\geq 0}$ and $(\beta(k))_{k\geq 0}$ each satisfy linear recurrences.  If there is a finite set of places $S$, including all infinite places, such that $\tau(k)\in \cO_K(S)^*$ for every $k\in\Z$, then both $\alpha(k)/\beta(k)$ and $\beta(k)/\alpha(k)$ are in the finitely generated ring $\cO_K(S)$ for every $k$.  So by the Hadamard quotient theorem~\cite{vdP}, the generating functions for the sequences are rational.  Hence both $\tau$ and $1/\tau$ are polynomials. So $\tau$ is constant.
\end{proof}

Recall that the \emph{congruence subgroup} of $\SL_\dim(\Z)$ determined by a positive integer $n$ is the finite index subgroup 
$$
\SL_\dim(\Z;n) := \{B\in \SL_\dim(\Z):B\equiv I \mymod n\}.
$$

\begin{lemma}
\label{lem:cosets}
Let $\rfn\colon\mats_\dim(\C)\dashrightarrow\mathbb{C}$ be a rational function, defined over $K$, that is regular and non-zero at every point in $\SL_\dim(\Z)$. Let $Y\in \SL_\dim(\Z)$ be a matrix. Assume $\abs[\place]{\rfn(Y)}>1$ for some place $\place\in M_K^{\mathrm{fin}}$, and let $p\in\Z_{\geq 0}$ be the unique prime for which $\abs[\place]{p}<1$.  Then there exists a positive integer $k$ such that $\abs[\place]{\rfn(YB)}>1$ for every $B\in \SL_\dim(\Z;p^k)$.
\end{lemma}

\begin{proof}
Let $\epsilon_1,\dots,\epsilon_s$ be an integral basis for $\cO_K$, and write $\sigma = \alpha / \beta$ as a quotient of coprime polynomials with coefficients in $K$.  We may rationalize $\sigma$ by multiplying the numerator and denominator by the non-trivial Galois conjugates of $\beta$; that is, the polynomials obtained by application of an element of ${\rm Gal}(\bar K/ K)$ to the coefficients of $\beta$.  Since $\beta(Y) \ne 0$ and $Y$ is defined over $\mathbb{Q}$, the rationalization is regular at $Y$ and has denominator with rational coefficients.  We can therefore write
$$\rfn = \sum \rfn_j\epsilon_j$$
where each $\rfn_j$ is a rational function with rational coefficients, regular at $Y$; cancelling denominators, we can assume these coefficients are integers. Writing $\rfn_j = \alpha_j/\beta_j$ as a quotient of coprime integer polynomials, we have  $\min_j\abs[p]{\beta_j(Y)}=p^{-m}$ for some $m>0$.

Take $k=2m+1$ and suppose $B\in \SL_\dim(\Z;p^k)$. For any $j$ we can write
$$
\rfn_j(YB) = \frac{\alpha_j(YB)}{\beta_j(YB)} = \frac{a_jp^k + \alpha_j(Y)}{b_jp^k + \beta_j(Y)}
$$
for some $a_j,b_j\in\Z$.
This gives
$$
\abs[\place]{\rfn(YB) - \rfn(Y)} \leq \max \abs[\place]{\epsilon_j} \abs[p]{\frac{a_jp^k\beta_j(Y) - b_j p^k \alpha_j(Y)}{(b_jp^k + \beta_j(Y))\beta_j(Y)} } \leq 1\cdot p^{-1} < 1,
$$
since $\abs[\place]{\epsilon_j} \leq 1$ and $p^k$ divides the numerator but not the denominator of the fraction accompanying $\epsilon_j$.  We conclude that $\abs[\place]{\rfn(YB)} = \abs[\place]{\rfn(Y)} > 1$.
\end{proof}

\begin{lemma}
\label{lem:cosets2}
Let $\rfn_1,\dots,\rfn_\ell\colon\mats_\dim(\C)\tto\C$ be rational functions defined over $K$, that are regular and non-zero at every point in $\SL_\dim(\Z)$. Let $Y\in \SL_\dim(\Z)$ be a matrix such that $\rfn_i(Y)\not\in \cO_K^*$ for every $i$.  Then there exists $n\in\Z_{\geq 0}$ such that for any $B\in \SL_\dim(\Z;n)$ and any $i$ we have $\rfn_i(YB)\notin\cO_K^*$.
\end{lemma}

\begin{proof}
After replacing some of the functions $\rfn_i$ by their reciprocals $\rfn_i^{-1}$ if necessary, we may assume that $\rfn_i(Y)\not\in \cO_K$ for $i=1,\ldots ,\ell$.  As in Lemma~\ref{lem:cosets}, we fix an integral basis $\epsilon_1,\ldots, \epsilon_s$ for $\cO_K$ and we decompose the functions $\rfn_i= \sum \rfn_{i,j}\epsilon_i$, with the rational functions $\sigma_{i,j}$ defined over $\mathbb{Q}$ and regular at $Y$. For each $i\in \{1,\ldots ,\ell\}$, $\rfn_i(Y)\notin \cO_K$, and thus there is some $j=j(i)$ such that $\rfn_{i,j}(Y) \notin \Z$.  We write $\rfn_{i,j} = \alpha_i/\beta_i$ where $\alpha_i$, $\beta_i$ are coprime integer polynomials in $\dim^2$ variables such that $\beta_i(Y)$ is non-zero.  By assumption $\alpha_i(Y)$ and $\beta_i(Y)$ are integers such that $\beta_i(Y)\nmid \alpha_i(Y)$.  We take $n = 2\prod_i \beta_i(Y)\in \mathbb{Z}\setminus \{0\}$.

Given $B\in \SL_\dim(\Z;n)$, we have (as in the proof of Lemma~\ref{lem:cosets}) integers $a_i,b_i$ such that
$$
\rfn_{i,j}(YB) = \frac{a_in + \alpha_i(Y)}{b_in + \beta_i(Y)},
$$
where the denominator is non-zero because $\beta_i(Y)$ divides $n/2$.  Since $\beta_i(Y)$ divides $n$ but not $\alpha_i(Y)$, it follows that $\rfn_{i,j}(YB)\notin \Z$ for $i=1,\ldots , \ell$ and $j=j(i)$.  Hence $\rfn_i(YB)\not \in \cO_K$ for $i=1,\ldots ,\ell$.    
\end{proof}

\begin{lemma}
\label{lem:twocosets}
If $n,n'\in\Z_{\geq 0}$ are coprime and $Y,Y'\in \SL_\dim(\Z)$ for $\dim\geq 3$, then the $Y$-coset of $\SL_\dim(\Z;n)$ intersects the $Y'$-coset of $\SL_\dim(\Z;n')$.
\end{lemma}

\begin{proof}
Since $\SL_\dim(\Z;n)$ and $\SL_\dim(\Z;n')$ are finite index normal subgroups of $\SL_\dim(\Z)$, so is the product $G:=\SL_\dim(\Z;n)\SL_\dim(\Z;n')$.  Since $\dim \geq 3$, $G$ is itself a congruence subgroup~\cite{BLS}, i.e. 
$G = \SL_\dim(\Z;n'')$ for some $n''>0$.  Since $\gcd(n,n') = 1$, we have $a,b\in\Z$ such that $an+bn'=1$.  So $G$ contains in particular the matrix $I + B_{12}=(I+anB_{12})(I-bn'B_{12})$, where $B_{12}$ is the matrix with $12$-entry equal to $1$ and all other entries equal to $0$.  Thus $n''=1$ and $G = \SL_\dim(\Z)$ is the entire group.  
It follows that $Y^{-1}Y' = B(B')^{-1}$ for some $B\in \SL_\dim(\Z;n)$ and some $B'\in \SL_\dim(\Z;n')$, giving us that $YB = Y'B' \in Y\, \SL_\dim(\Z;n)\cap Y'\, \SL_\dim(\Z;n') $.
\end{proof}

Putting the above results together, we arrive at the following summary statement.

\begin{corollary}
\label{cor:notintegers}
There exists a coset $\mathcal{Y}\subset \SL_\dim(\Z)$ of a finite index
subgroup such that the following hold for any $Y\in\mc {Y}$.
\begin{enumerate}
 \item $\rfn(Y,\cvec,\dvec)\notin \cO_K^*$ for any $\cvec\in\cvecs$ and $\dvec\in\dvecs$; and
 \item $\rfn(Y,\cvec,\dvec)/\rfn(Y,\tilde\cvec,\tilde\dvec)\notin \cO_K^*$ for any $\cvec,\tilde\cvec \in \cvecs$ and $\dvec,\tilde\dvec\in\dvecs$ unless both pairs of vectors are linearly dependent.  
\end{enumerate}
\end{corollary}

\begin{proof}
Let $\{\rfn_1,\dots, \rfn_L\}$ be the collection of all rational functions $\rfn_i\colon\mats_\dim(\C)\tto K$ obtained by setting $\rfn_i(Y) = \rfn(Y,\cvec,\dvec)$ or $\rfn_i=\rfn(Y,\cvec,\dvec)/\rfn(Y,\tilde\cvec,\tilde\dvec)$, where (in the latter case) both pairs $\cvec,\tilde\cvec\in\cvecs$ and $\dvec,\tilde\dvec\in\dvecs$ are linearly independent.    By Lemma~\ref{lem:nonconstant}, the $\rfn_i$ are all non-constant,
and they are regular and non-zero at any $Y\in\SL_\dim(\Z)$.

By Lemmas~\ref{lem:nonconst2} and~\ref{lem:manyprimes} there exists $Y_1\in \SL_\dim(\Z)$ such that $\sigma_1(Y_1)\notin \cO_K^*$.  So by Lemma~\ref{lem:cosets}, there exist $n_1\in\Z_{\geq 0}$ and 
$\rfn_1(Y_1B)\notin\cO_K^*$ for all $B\in\SL_\dim(\Z;n_1)$.  Suppose inductively that for some $\ell<L$ there exists $Y_\ell\in \SL_\dim(\Z)$ and $n_\ell\in\Z_{\geq 0}$ such that $\rfn_1(Y_\ell B),\dots,\rfn_\ell(Y_\ell B)\notin\cO_K^*$ for all $B\in\SL_\dim(\Z;n_\ell)$.

By Lemmas~\ref{lem:nonconst2},~\ref{lem:manyprimes} and~\ref{lem:cosets}, we may also choose a matrix $Y'\in \SL_\dim(\Z)$, a prime $p$ not dividing $n_\ell$, and $k\in\Z_{\geq 0}$, such that $\rfn_{\ell+1}(Y'B')\notin \cO_K^*$ for $B'\in \SL_\dim(\Z;p^k)$.  Finally, Lemma~\ref{lem:twocosets} tells us that the cosets $Y_\ell\cdot \SL_\dim(\Z;n_\ell)$ and $Y'\cdot \SL_\dim(\Z;p^k)$ intersect non-trivially.  So picking $Y_{\ell+1}$ in the intersection then gives $\rfn_i(Y_{\ell+1})\not\in \cO_K^*$ for all $i=1,\ldots ,\ell+1$.  Lemma~\ref{lem:cosets2} further yields an $n_{\ell+1}\in\Z_{\geq 0}$ such that $\rfn_i(Y_\ell B)\notin\cO_K^*$ for any $i=1,\dots,\ell+1$ and any $B\in\SL_\dim(\Z;n_{\ell+1})$.  Once $\ell+1$ reaches $L$, the induction is complete.
\end{proof}

\proofof{Theorem~\ref{thm:noresonances}} Let $\mc{Y}\subset \SL_\dim(\Z)$ be the coset given by Corollary~\ref{cor:notintegers} and let $Y\in\mc Y$ be any element.  If $a,b\in\Z$ and $t\in\discty(\gamma_Y)$ satisfy $a\theta = bt\,(\mymod 1)$, then we have $\cvec\in\cvecs$ and $\dvec\in\dvecs$ such that 
$$
(\maxeval/\bmaxeval)^a = e^{4\pi\imunit a\theta} = e^{4\pi\imunit bt} = \rfn(Y,\cvec,\dvec)^b.
$$
But $\maxeval$ and $\bmaxeval$ are eigenvalues of a matrix in $\SL_\dim(\Z)$ and therefore units of $\cO_K$.  So unless $b=0$, the equation implies that $\rfn(Y,\cvec,\dvec)$ is a unit in $\cO_K$, contrary to our choice of $Y$.  And if $b=0$, it follows that $a=0$ because by hypothesis no power of $\maxeval$ is real. 

Now suppose $a\theta = b(t-t')$ for some $a,b\in\Z$ and $t,t' \in \discty(\gamma_Y)$.  Then
$$
(\maxeval/\bmaxeval)^a = \left(\frac{\rfn(Y,\cvec,\dvec)}{\rfn(Y,\tilde\cvec,\tilde\dvec)} \right)^b
$$
for some $\cvec,\tilde\cvec\in\cvecs$ and $\dvec,\tilde\dvec\in\dvecs$.  If $b\neq 0$, then
$\rfn(Y,\cvec,\dvec)/\rfn(Y,\tilde\cvec,\tilde\rvec)$ is a unit in $\cO_K$ as before.  Then Corollary~\ref{cor:notintegers} tells us that $\tilde\cvec$ is a multiple of $\cvec$ and $\tilde\dvec$ is a multiple of $\dvec$. In this case,
$$
e^{4\pi\imunit t} = \rfn(Y,\tilde\cvec,\tilde\dvec)= \rfn(Y,\cvec,\dvec) = e^{4\pi\imunit t'},
$$
which implies that $2(t-t') = 0\mymod 1$ and also $(\maxeval/\bmaxeval)^{a} = 1$.  Hence $a=0$; and if $t-t'\neq 0\mymod 1$, then $b$ is even.
\qed

\begin{remark}
\label{rmk:practical}
The reduction of Theorem~\ref{thm:noresonances} to Corollary~\ref{cor:notintegers} furnishes a reasonably practical way to verify the conclusion of Theorem~\ref{thm:noresonances} for specific sets $\cvecs$ and $\dvecs$ and matrices $A$ and $Y$.  That is, from the given data, one generates finitely many elements $\rfn(Y)$, $\rfn(Y)/\tilde \rfn(Y) \in K$ which can then be checked very quickly by computer to see whether any are algebraic integers.  Implementing the check in software such as Maple, Mathematica and Sage requires only a few lines of code.
\end{remark}

\section{Proof of Theorem \ref{thm:mainb}}
\label{sec:proof}
\subsection{Setup.}
\label{sec:transetup}
Let us begin by recalling the relevant notation and assumptions from Theorem~\ref{thm:mainb}.  For convenience we take $K$ to be a number field containing the finitely many pertinent elements of $\bbq$ specified in the next couple of paragraphs, fixing an embedding $K\hookrightarrow\C$ and letting $|\cdot|$ denote the induced archimedean absolute value on $K$.   

We are given a (possibly transcendental) irrational number $\theta$, a vector $\rho := (\rho_1,\dots,\rho_\dim)\in K^\dim$ and a piecewise, but not globally, constant $1$-periodic vector-valued function $\gamma:\R\to K^\dim$.  
These satisfy the following additional conditions
\begin{enumerate}
 \item The entries of $\rho$ are pairwise multiplicatively independent, and each satisfies $|\rho_i|<1$;
 \item $\theta$ and $\rho$ are discordant (see the paragraph before Theorem~\ref{thm:mainb});
 \item \label{it:maximality} for each $j\in\Z_{\geq 0}$ sufficiently large, the function $t\mapsto \pair{\gamma(t)}{(\rho_1^j,\dots,\rho_\dim^j)}$ is real-valued and maximized by $t = j\theta$.  
\end{enumerate}
We aim to show that
$
\Omega = \sum_{j=1}^\infty \pair{\gamma(j\theta)}{\rho_1^j,\dots,\rho_\dim^j)}
$
is transcendental.  

To get a contradiction, we assume henceforth that $\Omega\in\bbq$.  We then let $Z' \subset \bbq$ be the finite set of values taken by the components of $\gamma$ together with  $\{ -1, 0, 1 \}$, and we let $Z = \left( Z' - Z' \right) \cup \{\Omega,\rho_1, \dots, \rho_\dim \}$ be the set of differences of elements of $Z'$ together with the finite set $\{\Omega,\rho_1, \dots, \rho_\dim \}$.  Enlarging if necessary, we may assume that $K$ contains $Z$.  For purposes of applying Theorem \ref{thm:Evertse} throughout this section, we let $S\subset M_K$ consist of all infinite places together with all finite places that have non-zero valuation on some element of the finite set $Z$, and we take $T = \{|\cdot|^2\}$. 

We will proceed with a Liouville-style argument, constructing high-quality but not exact algebraic approximations of $\Omega$.  The maximality hypothesis (iii) will allow us to rule out exactness.  We employ Theorem~\ref{thm:unitequations} to strengthen it, showing that $\pair{\gamma(t)}{(\rho_1^j,\dots,\rho_\dim^j)}$ is not only maximized by $t=j\theta$, but for the most part \emph{strictly} so.

\begin{lemma} \label{lem:fixedjpositive} For $j\in\Z_{\geq 0}$ sufficiently large, if $t\in\R$ with $\gamma(t) \ne \gamma(j\theta)$, then 
$$
\iprod{\gamma(j\theta) -\gamma(t)}{(\rho_1^j,\dots,\rho_\dim^j)} > 0.
$$
\end{lemma}

\begin{proof}
By the maximization hypothesis~\ref{it:maximality} on $\gamma$, it suffices to show that there are only finitely many $j \in \mathbb{N}$ for which there exists $t \in \R$ with $\gamma(j \theta) \ne \gamma(t)$ and
\begin{equation}
\label{eqn:orthogonal}
\iprod{\gamma(j \theta) -\gamma(t)}{(\rho_1^j,\dots,\rho_\dim^j)} = 0.
\end{equation}
Given such a $j$, assume without loss of generality that $\gamma_1(j \theta) \ne \gamma_1(t)$.  Rearranging (\ref{eqn:orthogonal}) we obtain
$$
\sum_{i = 2}^\dim \frac{\gamma_i(j \theta) - \gamma_i(t)}{\gamma_1(j \theta) - \gamma_1(t)} \cdot \left( \frac{\rho_i^j}{\rho_1^j} \right) = 1.
$$
So $z_i = (\rho_i/\rho_1)^j$ is a (possibly degenerate) solution of
$
\sum_{i=2}^\dim a_i z_i = 1,
$
where the coefficients $a_i$ are taken from the finite set 
$$
\left\{\frac{\gamma_i(s) - \gamma_i(t)}{\gamma_1(s)-\gamma_1(t)}:\gamma_1(s)\neq\gamma_1(t)\text{ and }i\neq 1\right\}
$$
and $z_i$ in the finite rank multiplicative group $G$ generated by the finitely many coordinates $\rho_i$ of $\rho$.  While the sum $\sum_{i = 2}^\dim a_i z_i = 1$
may contain a vanishing subsum, we may discard a maximal set $I\subset\{2,\dots,\dim-1\}$ for which $\sum_{i\in I} a_i z_i $ vanishes.  As the full sum is non-zero, $I$ is necessarily a proper subset, and we obtain a \emph{non-degenerate} solution $\sum_{i\notin I} a_iz_i = 1$.  Theorem~\ref{thm:unitequations} then tells us that there are only finitely many such solutions with all $z_i\in G$.  Hence there is a finite set $X\subset G$, independent of $j$, such that whenever~\eqref{eqn:orthogonal} holds, $z_i = (\rho_i/\rho_1)^j \in X$ for some $i\in \{2,\dots,\dim\}$.

On the other hand, the independence hypothesis (i) above implies for any $i\neq 1$ that distinct values of $j$ yield distinct elements $(\rho_i/\rho_1)^j\in G$.  In particular $(\rho_i/\rho_1)^j\in X$ for only finitely many $j$, and as there are only $\dim-1$ possibilities for $i$, we conclude that~\eqref{eqn:orthogonal} can hold for only finitely many $j$.
\end{proof}

\subsection{Convergents and $n$-irregular indices} \label{sec:convergents} 
For $t \in \mathbb{R}$, we let $\| t\|$ denote the distance from $t$ to the nearest integer and let $\{t\}\in[0,1)$ denote the fractional part of $t$.

Recall (from e.g. Chapters X-XI of~\cite{HW}) that any irrational number $\theta\in\R$ admits an infinite sequence of continued fraction approximants $m_i/n_i$, with $n_i$ strictly increasing, $m_i$ coprime to $n_i$, and $|n_i\theta-m_i|< \frac 1{n_i}$ for all $i\in\Z_{\geq 0}$.
\begin{definition}
  We call $m_i/n_i$ the \emph{convergents} of $\theta$, and write $\cvgts:=\{n_i\}_i$ for the set of convergent denominators of $\theta$.
\end{definition}
We recall here three elementary properties of convergents, proofs of which may be found in the first two chapters of~\cite{Khi64}. \begin{enumerate}
\item A convergent $m_i/n_i$ of $\theta$ is a best approximation of the second kind; that is, for all $n \in \mathbb{N}$, $\|n \theta\| < \|n_i \theta\|$ implies $n > n_i$.
\item If $m/n \in \mathbb{Q}$ is in lowest terms with $|\theta - m/n| < \frac{1}{2n^2}$, then $m/n$ is a convergent of $\theta$. It follows that if $n \in \mathbb{N}$ and $\| n \theta \| < \frac{1}{2n},$ then $n$ is a multiple of some element of $\cvgts$.
\item For any $i\in\mathbb{N}$ we have $\|n_i\theta\|<\frac1{n_{i+1}}$.
\end{enumerate}

\begin{definition} Given $n \in \cvgts$ and $j \in \mathbb{N}$ with $j > n$, 
we say that $j$ is \emph{$n$-irregular} if $\gamma((j-n) \theta ) \ne \gamma(j \theta )$.
\end{definition}

If $j$ is $n$-irregular, then $j\theta$ approximates some $t\in\discty(\gamma) \cap [0, 1)$ in the sense that
\begin{equation}
\|j\theta - t\| \leq \|n \theta\|,
\end{equation}
and we call this $t$ a \emph{crossing angle} of $j$ for $n$.  Since $\discty(\gamma) \cap [0, 1)$ is finite, when $\|n\theta\|$ is small enough, the crossing angle of $j$ for $n$ is unique.  In particular if $n\in\cvgts$ is a sufficiently large convergent denominator, the third property of convergents noted above yields a unique crossing angle of $j$ for $n$ whenever $j$ is $n$-irregular.

We now show that the approximability properties of convergents of $\theta$ ensure that $n$-irregular indices are sparse for $n \in \cvgts$.

\begin{lemma} \label{lem:irregularsparse} Let $D$ be the number of elements of $\discty(\gamma) \cap [0, 1)$.  Given $C > 0$ and $n\in\cvgts$ sufficiently large, there are at most $D(C-1)$ $n$-irregular indices $j\in (n, Cn]$.  
\end{lemma}

\begin{proof} Suppose that $n \in \cvgts$ is large enough that crossing angles are well-defined for $n$-irregular indices.  Since $\theta$ is irrational, we may also assume that for $j > n$ we have $\{j \theta\} \not\in \discty(\gamma),$ as only finitely many $j \in \mathbb{N}$ will fail this condition.  Suppose that $j \ne j'$ are $n$-irregular indices with the same crossing angle $t\in\discty(\gamma)$.  Since $n \in \cvgts$ is a large convergent of $\theta$, $n \theta \mod 1$ is either a small positive number or slightly less than $1$.   In the former case, any $n$-irregular index $j$ has $j\theta \mod 1$ slightly larger than $t$, and in the latter, slightly less than $t$ (unless $t = 0$, in which case any $n$-irregular index $j$ has $j \theta \mod 1$ close to $1$).  Thus $j \theta \mod 1$ and $j' \theta \mod 1$ are either both smaller or larger than $t$; since $\|j \theta - t\|,\|j'\theta - t\| < \|n \theta\|$, we see that $\|(j - j') \theta\| < \|n \theta\|$.  Since $n$ is a convergent denominator, we conclude that $|j - j'| > n$.  Thus $(n, Cn]$ contains at most $(C-1)$ $n$-irregular numbers with crossing angle $t$.  The assertion then follows from applying this bound to all possible crossing angles $t\in\discty(\gamma) \cap [0, 1)$.  
\end{proof}

\subsection{Ubiquity of $n$-irregular indices}\label{sec:ubiquity}
We now define some useful algebraic approximations of $\Omega$ and estimate their quality.  A consequence of our estimates and the assumption $\Omega \in K$ will be that, in spite of Lemma \ref{lem:irregularsparse}, $n$-irregular numbers occur with some frequency.

Define $\omega^{(i)}:= \sum_{j=1}^\infty \gamma_i(j\theta)\rho_i^j$, so that $\Omega = \sum_{i=1}^\dim \omega^{(i)}$.  To each pair of integers $b \geq 0$ and $n\geq 1$, we associate the following $K$-rational approximation of $\omega^{(i)}$, which has $n$-periodic coefficients after the first $bn$ terms:
$$
\omega_{n,b}^{(i)} := \sum_{j=1}^{bn} \gamma_i( j \theta) \rho_i^{j} + \frac{1}{1 - \rho_i^n} \sum_{j=bn+1}^{(b+1)n} \gamma_i( j \theta ) \rho_i^{j}.
$$
Then
\begin{equation} \label{eqn:tildeexpansion}
\omega^{(i)} - \omega_{n,b}^{(i)} =
\sum_{j > (b+1)n} (\gamma_i(j\theta) - \gamma_i(\tilde j\theta))\rho_i^j,
\end{equation}
where $\tilde j\in (bn,(b+1)n]$ agrees with $j$ modulo $n$.  Finally, write
\begin{equation}\label{eq:Onb}
\Omega_{n, b} := \sum_{i=1}^\dim \omega_{n,b}^{(i)}.
\end{equation}
When $n\in\cvgts$ is large, $\Omega_{n, b}$ is a real number by~\ref{it:maximality} in \S\ref{sec:transetup}, and the sequence $(\gamma_i(j\theta))_{j\in\Z_{\geq 0}}$ is nearly $n$-periodic, and so the difference $\Omega - \Omega_{n, b}$ is a small real number.  Crucially, however, it does not vanish.

\begin{proposition}
\label{prop:underapprox}
Let $b$ be a positive integer.  Then $
\Omega > \Omega_{n, b}
$ for all but finitely many $n \in \cvgts$. 
\end{proposition}

\begin{proof}
By Equations~\eqref{eqn:tildeexpansion} and~\eqref{eq:Onb} we have
$$
\Omega - \Omega_{n,b} = \sum_{i=1}^\dim \omega^{(i)} - \omega_{n,b}^{(i)} = \sum_{j > (b+1)n} \iprod{\gamma(j\theta) -\gamma(\tilde j\theta)}{(\rho_1^j,\dots,\rho_\dim^j)}.
$$
By Lemma~\ref{lem:fixedjpositive} and the maximization hypothesis (see~\ref{it:maximality} in \S\ref{sec:transetup}) on $\gamma$, it suffices to find for $n \in \cvgts$ sufficiently large a single $j\geq (b+1)n$ such that $\gamma(j\theta) \ne \gamma(\tilde j \theta)$.

As $\gamma$ is non-constant and piecewise constant, there exist non-empty open intervals $I,I'\subset [0,1)$ such that $\gamma$ is constant on $I$ and $I'$ but $\gamma(I)\neq \gamma(I')$.  As $\theta\notin\Q$, we have $\{p\theta\} \in I'$ for some $p\in\Z_{\geq 0}$.  Now for $n\in\cvgts$ large enough, we have
\begin{itemize}
 \item $n>p$,
 \item $\{(p+bn)\theta\} \in I'$, and 
 \item $\|n\theta\|$ is smaller than the width of $I$.
\end{itemize}
The last condition guarantees the existence of an integer $k'>b$ such that $\{(p+k'n)\theta\}\in I$.
Taking $j=p+k'n$ we have $\{j\theta\} \in I$ and, by the first two conditions, that $\tilde j = p +bn$, so that $\{\tilde j\theta\} \in I'$.  Thus, $\gamma(j\theta) = \gamma(I) \neq \gamma(I') = \gamma(\tilde j \theta)$ as desired.
\end{proof}

For each $n\in\Z_{\geq 0}$, write $$Q_n = \prod_{i=1}^\dim (1-\rho_i^n).$$  Note that $Q_n$ and $Q_n \omega_{n,b}^{(i)}$ are polynomials in $\rho$ of degree at most $\dim n$ and $(\dim +b)n$, respectively, with coefficients in the finite set $Z$.  Furthermore, as all elements of $Z$ are $S$-units by definition of $S$, $Q_n$ and $Q_n \omega_{n,b}^{(i)}$ are $S$-integral.  We have
\begin{equation}
\label{eqn:tailformula}
Q_n\cdot  \left(\omega^{(i)} - \omega_{n,b}^{(i)}\right) = \prod_{k \neq i}\left(1-\rho_{k}^n\right)\sum_{j > (b+1)n} \left(\gamma_i(j\theta) - \gamma_i((j-n)\theta)\right)\rho_i^j.
\end{equation}

Recall now the Vinogradov notation `$\gg$' introduced after Theorem \ref{thm:Evertse}.

\begin{proposition} \label{prop:nottoosmall}  Let $b$ be a positive integer and let $R>0$ be as in Lemma~\ref{lem:heightbound}.  Then 
for $n\in\cvgts$ we have,
$$
|Q_n\cdot \left(\Omega-\Omega_{n,b}\right)|^2 \gg R^{-3(b+\dim)n}.
$$
\end{proposition}

\begin{proof} Assume first that $\Omega \ne 0$.  Then both $z_1 := Q_n\Omega$ and $z_2  := - Q_n\Omega_{n,b}$ are non-zero $S$-integers, and $|z_i-\Omega|\to0$ as $\cvgts\ni n\to\infty$, and so for sufficiently large $n$, we have $z_1,z_2\ne0$.   Then Proposition~\ref{prop:underapprox} ensures that
$z_1 + z_2 \neq 0$ for sufficiently large $n\in \cvgts$.  

Since $z_1$ and $z_2$ are polynomials in $\rho$ of degree at most $(\dim +b)n$ with coefficients in the finite set $Z$, we may apply Lemma~\ref{lem:heightbound} to obtain the bound $\prod_{\place \in M_K} \max \{ |z_j|_\place, 1 \} \ll R^{(b+\dim)n}$ for $j \in \{1, 2\}$.  Hence 
$$
\prod_{\place \in S} |z_j|_\place \ll R^{(b+\dim)n},
$$ 
and
$$
H_S(z_1, z_2) = \prod_{\place \in S} \max \{ |z_1|_\place, |z_2|_\place \} \leq \prod_{\place \in S} \prod_{j = 1}^{2} \max \{ |z_j|_\place, 1 \} \ll R^{2(b+\dim)n}.
$$
We conclude from Theorem~\ref{thm:Evertse}, taking $\mathbf{z} = (z_1,z_2)$ and $\epsilon=1$, that 
$$
|z_1 + z_2|^2 \gg \frac{|z_1|^2}{R^{3(b+\dim)n}} \gg \frac{1}{R^{3(b+\dim)n}},
$$
where the second inequality follows from the assumption that $\Omega \ne 0$ and so $|z_1|$ is bounded below independently of $n$.

In the case that $\Omega = 0,$ choose $z \in Z \setminus \{ 0 \}$, and replace $z_1$ and $z_2$ with $z_1+z$ and $z_2-z$ to guarantee that neither vanishes for large $n$; with this choice, $z_1$ and $z_2$ remain $\rho$-polynomials of the same degree, with coefficients in $Z$, and $z_1 + z_2$ is unchanged.  Proceeding as above we deduce the desired inequality.
\end{proof}

\begin{corollary} \label{cor:nobiggaps} There exists $c_0>1$ such that for all $M \in \mathbb{N}$, there is at least one $n$-irregular number $j\in [Mn,c_0Mn)$ for all $n\in\cvgts$ sufficiently large.
\end{corollary}

\begin{proof} 
Choose $\lambda\in(0,1)$ such that $|\rho_i|\leq\lambda$ for all $1\leq i\leq \dim$. Suppose for some integer $C\geq 1$ and arbitrarily large $n\in\cvgts$ that there are no $n$-irregular $j\in (Mn,CMn]$. Then we estimate using~\eqref{eqn:tailformula} and $b = M$ that
\begin{eqnarray*}
|Q_n \left(\Omega - \Omega_{n,M}\right)| & = & \left|\sum_{i=1}^\dim Q_n \left(\omega^{(i)} - \omega_{n,M}^{(i)}\right)\right| \\
& \leq & \sum_{i=1}^\dim \prod_{k\neq i} |1-\rho_k^n| \sum_{j> M(n+1)}|\gamma_i(j\theta) - \gamma_i((j-n)\theta)||\rho_i^j| \\
& \ll &\sum_{i=1}^\dim \sum_{j > CMn}|\gamma_i(j\theta) - \gamma_i((j-n)\theta)||\rho_i^j| \\
& \ll &
\lambda^{CMn}.
\end{eqnarray*}
So from Proposition~\ref{prop:nottoosmall}, we infer that
$
\lambda^{2CMn} \gg R^{-3(M+\dim)n}.
$
This implies that $C< (\dim+1)\frac{3\log R}{2\log\lambda^{-1}}$, and so picking $c_0>\lambda\left(1,(\dim+1)\frac{3\log R}{2\log\lambda^{-1}}\right)$, we obtain the desired result.
\end{proof}

\begin{remark} If $\theta$ is sufficiently well-approximable by rationals (for example, if $\theta$ has unbounded integers in its continued fraction expansion), one may prove that for any $C \in \Z_{\geq 0}$, there are infinitely many convergent denominators such that $[2n, Cn)$ has no $n$-irregular indices.  This together with Corollary~\ref{cor:nobiggaps} provides an immediate contradiction and so implies transcendence of $\Omega$.  As we do not want to impose any approximability constraints on $\theta$, we proceed with a more delicate argument that applies to general $\theta$.
\end{remark}

\subsection{Residual sums}\label{sec:nonvanishing}
For the remainder of this section we fix $b = 0$, writing
$$
\omega_n^{(i)} := \omega_{n, 0}^{(i)} = \frac{1}{1 - \rho_i^n} \sum_{j=1}^{n} \gamma_i( j \theta ) \rho_i^{j},
\quad\text{and}\quad
\Omega_n := \Omega_{n, 0} = \sum_{i=1}^\dim \omega_n^{(i)}.
$$
While we consider the sum only in the case $b = 0$, note that the results of Section~\ref{sec:ubiquity} are required with arbitrary values of $b$ below in Lemma~\ref{lem:weightedirregularsparse} and Corollary~\ref{cor:weightednobiggaps}. 

We wish to write $Q_n\cdot (\Omega - \Omega_{n})$ as a sum of $\rho$-monomials.  We have by Equation~\eqref{eqn:tailformula} that 
$$
Q_n\cdot (\Omega - \Omega_{n})= \sum_{i = 1}^\dim \left( \prod_{k \neq i}\left(1-\rho_{k}^n\right) \right) \left( \sum_{j > n} \left(\gamma_i(j\theta) - \gamma_i((j-n)\theta)\right)\rho_i^j \right).
$$
Recall the multi-index notation~\eqref{eqn:multiindex}.  Given $\alpha = (a_1, \dots, a_\dim) \in \Z_{\ge0}^\dim,$ define $\zeta_n(\alpha)\in Z$ to be the coefficient of the $\rho$-monomial $\rho^{\alpha}=\rho_1^{a_1} \cdots \rho_\dim^{a_\dim}$ in this series; that is,
\begin{equation} \label{eqn:definezeta}
Q_n\cdot (\Omega - \Omega_{n}) = \sum_{\alpha\in \Z_{\ge0}^\dim} \zeta_n(\alpha)\rho^{\alpha},\qquad \zeta_n(\alpha)\in Z\ {\rm for}\ \alpha\in \Z_{\ge0}^\dim.
\end{equation}

By the discussion preceding Lemma~\ref{lem:irregularsparse}, $\zeta_n(\alpha) \ne 0$ if and only if there exists some component $a_i$ of $\alpha$ such that $a_k \in \{0, n\}$ for all $k \ne i$, $a_i > n,$ and 
$$
\gamma_i(a_i \theta) - \gamma_i((a_i-n)\theta) \ne 0.
$$
In this case, we have 
$$
\zeta_n(\alpha) = (-1)^{\frac{1}{n} \sum_{k \ne i} a_k} \left( \gamma_i(a_i \theta) - \gamma_i((a_i-n) \theta) \right).
$$

\begin{definition} \label{def:residual} {\em We say a multi-index $\alpha \in \Z_{\ge0}^\dim$ is \emph{$n$-residual} if $\zeta_n(\alpha) \ne 0$.  If $\alpha = (a_1, \dots, a_\dim)$ is $n$-residual with $a_i > n$, we say that $a_i \mathbf{e}_i$ is the \emph{irregular component} of $\alpha$, where $\mathbf{e}_i\in\Z^\dim$ denotes the $i$-th vector of the standard basis.} \end{definition}

Our argument will rely on the fact that due to the discordance condition, when $n \in \cvgts$ is odd, subsums of (\ref{eqn:definezeta}) which are non-degenerate (in the sense of Theorem~\ref{thm:unitequations}) and have small support
can vanish only in very limited circumstances.   We make this statement precise in the following proposition.  Note that as successive elements of $\cvgts$ are coprime, there are infinitely many odd elements of $\cvgts$.

\begin{proposition}
\label{prop:nonvanishingsubsums}
Let $L$ be a natural number.  Then there exists a positive integer $N=N(L)$ such that whenever $n\in\cvgts$ is odd and larger than $N$, and $\mathcal{A}\subseteq \Z_{\ge0}^\dim$ is a set of size at most $L$ such that 
$$
\sum_{\alpha \in \mc A}\zeta_n(\alpha) \rho^\alpha = 0
$$ 
is a vanishing non-degenerate subsum of (\ref{eqn:definezeta}), we have that every $\alpha \in \mc A$ has the same irregular component.
\end{proposition}

\begin{proof}
Let $G$ be the multiplicative subgroup of $\C^*$ generated by the non-zero elements of $Z$, and fix an isomorphism $G \rightarrow \Z^r \oplus G_0$, where $G_0$ is the torsion subgroup.  Projecting onto the first factor gives a (surjective) homomorphism $\pi: G \rightarrow \Z^r$ with $\ker \pi = G_0$.  The coordinates $\rho_i$ of $\rho$ are pairwise multiplicatively independent elements of $Z$ by hypothesis, so the vectors $v_i := \pi(\rho_i)$ are pairwise linearly independent.  

Suppose the proposition fails.  Then there is an infinite set $\mc N$ of odd $n \in\cvgts$ and for each $n\in\mc N$ a set $\mc A_n$ of size at most $L$ such that $\sum_{\alpha\in\mc A_n}\zeta_n(\alpha) \rho^\alpha$ is a vanishing non-degenerate $n$-residual subsum which, by the discussion preceding Definition~\ref{def:residual}, includes multi-indices $\beta = j\mathbf{e}_i+n\delta$ and $\beta' = j'\mathbf{e}_{i'} + n\delta'$ with different irregular components $j\mathbf{e}_i \neq j'\mathbf{e}_{i'}$, where $i, i' \in \{1, \dots, \dim\}$ and $\delta, \delta' \in \{0, 1\}^\dim$.  By refining the set $\mc N$, we may assume that $i,i',\delta,\delta'$ are fixed.  While the irregular indices $j,j' > n$ must vary with $n$, we can again refine to assume that their crossing angles $t,t'\in\discty(\gamma)$ do not.  

Fixing $n\in \mc N$, we rearrange the equation $\sum_{\alpha\in\mc A_n} \zeta_n(\alpha) \rho^\alpha = 0$ to get
$$
1 = \sum_{\alpha\in \mc A\setminus\{\beta\}} -\frac{\zeta_n(\alpha)}{\zeta_n(\beta)}\rho^{\alpha - \beta}.
$$
The sum on the right remains non-degenerate, with coefficients $-\zeta_n(\alpha) / \zeta_n(\beta)$ in a finite set that is independent of $n$.  Thus applying Theorem~\ref{thm:unitequations} with $G$ as above, we obtain a finite set $X\subset G$, independent of $n$, such that $\rho^{\alpha-\beta}\in X$ for all $n \in \mc N$ and all $\alpha,\beta\in \mc A_n$.  In particular, $\rho^{\beta - \beta'} \in X$.  Refining $\mc N$ still further, we may suppose that $\rho^{\beta' - \beta} \in X$ is the same for all $n\in\mc N$.  Applying the homomorphism $\pi:G\to \Z^r$ to $\rho^{\beta - \beta'}$, we obtain
\begin{equation}
\label{eqn:imagevector}
j'v_{i'} - j v_i + nw = u
\end{equation}
for some fixed $w,u\in \Z^r$ and all $n\in\mc N$.  Note our notation suppresses dependence of $j, j'$ on $n$.  From here we divide the argument into two cases.

Suppose first that $i\neq i'$, in which case there is a vector $\rvec\in \Z_{\ge0}^\dim$ such that $\pair{\rvec}{v_i} \neq 0 = \pair{\rvec}{v_{i'}}$.  Multiplying both sides of~\eqref{eqn:imagevector} by $\rvec$ then gives
$$
aj = bn + c 
$$
for some fixed $a,b,c\in\Z$ with $a\neq 0$; without loss of generality, we may cancel common factors to assume $(a, b, c) = 1$.  By $n$-irregularity we have $||j \theta - t|| < ||n \theta||$, so multiplying through by $\theta$ and letting $n\to\infty$ in $\mc N$ gives
$$
at \equiv c\theta\mymod 1.
$$
As $\theta$ and $\discty(\gamma)$ are discordant, $c = 0$, and either $a$ is even or $t = 0$.  Since $(a, b, c) = 1$ and $c = 0$, if $a$ is even then $aj = bn$ implies that $n$ is even as well, contradicting our assumption that $\mc N$ consists of odd $n \in \cvgts$. So $t=0$, and since $t$ is the crossing angle of $j$ for $n$, $\| j \theta \| < \| n \theta \|$.  As $j > n$ and $aj = bn$, we have $|a| < |b|$.  On the other hand, for large $n \in \mathcal{N}$, 
$$
|a|\norm{n\theta} > |a| \norm{j \theta} = \norm{a j \theta} = \norm{bn\theta} = |b|\norm{n\theta},
$$
so $|a| > |b|$, a contradiction.

Now suppose instead that $i=i'$ in~\eqref{eqn:imagevector}: we proceed similarly in this case.  Applying $\pi$ now yields
$$
(j-j') v_i = u - nw,
$$
where $j-j' \neq 0$ by our choice of $\beta,\beta'$ and $v_i\neq 0$ by linear independence of the $v_i$.  Hence we may restrict to a single coordinate of $v_i,u,w$ to get
$$
(j -j')a = c+nb
$$
where $a,b,c\in\Z$ and $a\neq 0$; cancelling common factors, we may assume $(a, b, c) = 1$.  Multiplying through by $\theta$ and letting $n\to\infty$ in $\mc N$ now gives
$$
a(t-t') \equiv c\theta \mymod 1
$$
By discordance, $c = 0$ and either $a$ is even or $t = t'$.

If $a$ is even, then by coprimality, $b$ is odd.  As each $n \in \mc N$ is also odd and 
$$
(j -j')a = nb,
$$
we have a contradiction.  So $t=t'$, and we have as in the proof of Lemma~\ref{lem:irregularsparse}
that $0 < \norm{(j - j')\theta} < \norm{n\theta}$.  As noted in~\ref{sec:convergents}, $n$ is a best approximation of the second kind, so we must have $|j -j'| > n$ and
therefore $|a| < |b|$. On the other hand, for large $n \in \mathcal{N}$,
we obtain
$$
|a|\norm{n\theta} > |a| \norm{(j-j')\theta} = \norm{a(j-j')\theta} = \norm{bn\theta} = |b|\norm{n\theta},
$$
and thus $|a|>|b|$, a contradiction.
\end{proof}

Despite the non-vanishing result of Proposition~\ref{prop:nonvanishingsubsums}, it is possible that some subsum of terms with the same irregular component vanishes.  However, a non-trivial sum containing \emph{all} of the $n$-residual terms which come from a fixed $n$-irregular number $j$ cannot vanish, as we now explain.  For $i \in \{1, \dots, \dim\},$ $n\in\cvgts$, and $j > n$ $n$-irregular, let $\mc R_{i, n,j}$ denote the set of $n$-residual $\alpha\in \Z_{\ge0}^\dim$ which have irregular component $j\mathbf{e}_i$.  

\begin{corollary} \label{cor:nonvanishing} Let $L$ be a positive integer.  Then there exists $N=N(L)$  such that whenever $n \in \cvgts$ is odd and larger than $N$ and $\mc A \subset \Z_{\ge0}^\dim$ is a set of size at most $L$ with the property that there exists an $n$-irregular index $j >n$ such that $\mc R_{i, n, j} \subset \mc A$ for all $i \in \{ 1, \dots, \dim \},$ we have 
$$
\sum_{\alpha \in \mc A} \zeta_n(\alpha) \rho^{\alpha} \ne 0.
$$
\end{corollary}

\begin{proof} Suppose towards a contradiction that $\sum_{\alpha \in \mc A} \zeta_n(\alpha) \rho^{\alpha} = 0$.
Working inductively, decompose $\mc A$ into a disjoint union $\mc A_1, \dots, \mc A_r$ of subsets so that 
$$
\sum_{\alpha \in \mc A_{\ell}} \zeta_n(\alpha) \rho^{\alpha} = 0
$$
is a non-degenerate vanishing subsum of $\mc A$ for each $\ell \in \{ 1, \dots, r \}$.  With $N = N(L)$ of Proposition~\ref{prop:nonvanishingsubsums}, for $n \in \cvgts$ odd and larger than $N$, we have for each $\ell \in \{ 1, \dots, r \}$ some $i(\ell)$ and $j(\ell)$ such that $\mc A_{\ell} \subset \mc R_{i(\ell), n, j(\ell)}$.  Since $j$ is $n$-irregular, we have $\gamma_{i_0}(j\theta) \ne \gamma_{i_0}((j-n)\theta)$ for some choice of $i_0 \in \{ 1, \dots, \dim\}$, so $\mc R_{i_0, n, j} \subset \mc A$ is non-empty. Therefore, $\mc R_{i_0, n, j}$ is a disjoint union of elements of a subset of $\{\mc A_1,\ldots , \mc A_r\}$, and so 
\begin{equation}\label{eqn:sumzero}
\sum_{\alpha \in \mc R_{i_0, n, j}} \zeta_n(\alpha) \rho^{\alpha} = 0.
\end{equation}

On the other hand, by~\eqref{eqn:tailformula} we have
$$
\sum_{\alpha \in \Z_{\ge0}^\dim} \zeta_n(\alpha) \rho^\alpha =  \sum_{i = 1}^\dim \left( \prod_{k \neq i}\left(1-\rho_{k}^n\right) \right) \left( \sum_{j > n} \left(\gamma_i(j\theta) - \gamma_i((j-n)\theta)\right)\rho_i^j \right).
$$
The terms of the sum on the right-hand side which correspond to $\alpha \in  \mc R_{i_0,n,j}$ are precisely those with exponent $j$ for $\rho_{i_0}$, so 
$$
\sum_{\alpha \in \mc R_{i_0, n, j}} \zeta_n(\alpha) \rho^\alpha = \left( \prod_{k \neq i_0}\left(1-\rho_{k}^n\right) \right) \left( \gamma_{i_0}(j\theta) - \gamma_{i_0}((j-n)\theta)\right)\rho_{i_0}^j.
$$
The right-hand expression above is non-zero as we have chosen $i_0$ so that $\gamma_{i_0}(j\theta) \ne \gamma_{i_0}((j-n)\theta)$. This contradicts~\eqref{eqn:sumzero}.
\end{proof}

\subsection{Completing the proof}\label{sec:fixthisname}

To any multi-index $\alpha \in \Z_{\ge0}^\dim$, we assign the (weighted) norm 
\begin{equation} \label{eqn:weightednorm}
\| \alpha \|_{\rho} := - \log |\rho^{\alpha}|.
\end{equation}
As $0 < |\rho_i| < 1$ for all $i$, any $n$-residual $\alpha$ with irregular component $j {\bf e_i}$ has norm $\| \alpha \|_{\rho}$ that is multiplicatively comparable to $j$ with constants independent of $n$.  We may therefore reformulate weighted versions of Lemma~\ref{lem:irregularsparse} and Corollary~\ref{cor:nobiggaps} as follows.

\begin{lemma} \label{lem:weightedirregularsparse}
There exists a constant $c_0 > 0$ such that for any $C > 0$, there are at most $c_0 C$ $n$-residual multi-indices $\alpha \in \Z_{\ge0}^\dim$ satisfying $\| \alpha \|_{\rho} \leq C n$ for $n \in \cvgts$ sufficiently large.
\end{lemma}

\begin{corollary} \label{cor:weightednobiggaps}There exists a constant $c_1 > 1$ such that for any $C>0$, the following holds for $n \in \cvgts$ sufficiently large: there is an $n$-irregular $j$ such that for all $i \in \{ 1, \dots, \dim \}$, the set of multi-indices satisfying $C n \leq \| \alpha \|_{\rho} < c_1 C n$ includes the set $\mc R_{i, n,j}$ of $n$-residual $\alpha$ with irregular components $j {\bf e_i}$.  
\end{corollary}

We now fix constants $\rhoconst, \cone, \ctwo,$ and $L$ as follows.  With $\theta, \rho,$ and $\gamma$ as in Theorem~\ref{thm:mainb}, choose $\rhoconst > \max_{1\le i\le \dim} \log\frac1{|\rho_i|}$.  Let $\cone > \max \{c_1,  (1+2\rhoconst \dim) \log R\}$ with $c_1$ as in Corollary~\ref{cor:weightednobiggaps}, and $\ctwo > 2c_1\cone$. Given such a choice of $\ctwo,$ by Lemma~\ref{lem:weightedirregularsparse}, there is a constant $L$ such that the sum
$$
\sum_{\| \alpha \|_{\rho} \leq \ctwo n} \zeta_n(\alpha) \rho^{\alpha}
$$
has at most $L$ non-zero terms.  We now bring together the technical details of the preceding subsections to ensure non-vanishing of well-chosen subsums.

\begin{lemma} \label{lem:goodKs} Given $n\ge1$, and suppose $\firstgap$ and $\secondgap$ satisfy $\cone n \leq \firstgap < 2 \cone n$ and $\ctwo n \leq \secondgap$. Then for $n \in \cvgts$ odd and sufficiently large,
$$
\sum_{\| \alpha \|_{\rho} \leq \firstgap} \zeta_n(\alpha) \rho^{\alpha} \ne 0
$$
and
$$
\sum_{\firstgap \leq \| \alpha \|_{\rho} \leq \secondgap } \zeta_n(\alpha) \rho^{\alpha} \ne 0.
$$
\end{lemma} 

\begin{proof}
By hypothesis and Corollary~\ref{cor:weightednobiggaps}, we have for large $n\in\cvgts$ that there is some $n$-irregular $j$ such that the multi-index set $\{\|\alpha\|_{\rho} \leq \firstgap\}$ includes $\mc R_{i, n,j}$ for all $i \in \{ 1, \dots, \dim \}$.  Thus by Corollary~\ref{cor:nonvanishing}, $\sum_{\| \alpha \|_{\rho} \leq \firstgap} \zeta_n(\alpha) \rho^{\alpha} \ne 0$ for $n \in \cvgts$ odd and sufficiently large.  The same argument applies to the second sum.
\end{proof}

As $L$ is fixed, there exists $\delta_1 >0$ such that for all $n \in \cvgts$ sufficiently large, there is a subinterval of $(\cone n, 2\cone n]$ with length at least $\delta_1 n$ that contains no number of the form $\| \alpha \|_{\rho}$ with $\zeta_n(\alpha) \ne 0$.  For each such $n \in \cvgts$, let $\firstgap \in (\cone n, 2\cone n]$ be the left endpoint of this interval. 

Fix $n \in \cvgts$, and let 
\begin{equation} \label{eq:x1x2}
z_1 := Q_n \Omega, \qquad z_2 = -Q_n \Omega_n,
\end{equation}
 and write 
$$
z_1 + z_2 = Q_n (\Omega - \Omega_n) = \sum_{\alpha \in \Z_{\ge0}^\dim} \zeta_n(\alpha) \rho^{\alpha}.
$$
Here $z_1$ and $z_2$ depend on $n$, but this is suppressed in the notation.

By Lemma~\ref{lem:goodKs}, the sum
$$\sum_{\| \alpha \|_{\rho} \leq \firstgap} \zeta_n(\alpha) \rho^{\alpha}$$
is non-zero,  and it contains at most $L$ terms.  We write
\begin{equation}\label{eqn:monomials}
z_3 + z_4 + \cdots + z_{\ell(n)} = -\sum_{\| \alpha \|_{\rho} \leq \firstgap} \zeta_n(\alpha) \rho^{\alpha} \ne 0,
\end{equation}
with the $z_i$ chosen to be the monomials of the right-hand sum which remain after removing a maximal vanishing subsum, noting that since $p(n) < 2k_1n<k_2n,$ $\ell(n) \leq L$ by the choice of $k_2$ above.

\begin{lemma} \label{lem:finalnonvanishing} For odd $n \in \cvgts$ sufficiently large, $z_1 + z_2 + \cdots + z_{\ell(n)}$ contains no vanishing subsum.
\end{lemma}

\begin{proof}  Suppose the lemma fails.  Remove a maximal vanishing subsum from $z_1 + \cdots + z_{\ell(n)}$ to obtain a minimal non-empty index set $I \subset \{ 1, \dots, \ell(n) \}$ such that $z_1 + \cdots + z_{\ell(n)} = \sum_{k \in I} z_k$.  First suppose that $I$ is non-empty. As in the proof of Proposition~\ref{prop:nottoosmall}, if $\Omega = 0$, replace $z_1, z_2$ with $z_1 + 1, z_2 - 1$.  Then we have $|z_1|$ and $|z_2|$ larger than a positive constant for all $n$ sufficiently large, while the terms $z_3, \dots, z_{\ell(n)}$ converge to $0$ as $n \in \cvgts$ goes to infinity, so for $n \in \cvgts$ sufficiently large, $I$ either contains both $z_1$ and $z_2$, or neither.  As $z_3 + \cdots + z_{\ell(n)}$ has no vanishing subsums,  $I$ contains neither $z_1$ nor $z_2$.  In particular, all terms of $I$ are $S$-units.

By Theorem~\ref{thm:Evertse} 
and the product formula, for any $\epsilon > 0$ we have
$$
\big| \sum_{k \in I} z_k \big|^2 \geq c(\epsilon, L) \frac{\max_{k \in I} |z_k|^2}{H_S(\mathbf{z})^{\epsilon}},
$$
where $\mathbf{z}=(z_k)_{k\in I}$. Each $z_k$ with $k \in I$ has $\rho$-degree bounded above by $c_3 \firstgap$ for some constant $c_3 = c_3(\rho) \geq  1$.  By Lemma~\ref{lem:heightbound} it follows that
$$
H_S({\mathbf{z}}) \ll R^{c_3 \firstgap L}.
$$
As $\| \alpha \|_{\rho} \leq \firstgap$ for all elements contributing to the sum, $\max_{k \in I} |z_k| \gg e^{-\firstgap}$.  So 
$$
\big| \sum_{k \in I} z_k \big|^2 \gg \frac{e^{-2\firstgap}}{R^{c_3 L \firstgap \epsilon}}.
$$
On the other hand, since $\firstgap$ was chosen to be the left endpoint of an interval of length $\delta_1n$ with $\zeta_n(\alpha) = 0$ whenever $\| \alpha \|_{\rho}$ lies in the interval, we have
$$
\big| \sum_{k \in I} z_k \big|^2 = \bigg| \sum_{\| \alpha \|_{\rho} > \firstgap + \delta_1 n} \zeta_n(\alpha) \rho^{\alpha} \bigg|^2 \ll e^{-2\firstgap - 2\delta_1 n}.
$$
Combining the estimates, 
$$
e^{2\delta_1 n} \ll R^{c_3 L \firstgap \epsilon} \ll R^{2 c_3 L \cone n \epsilon},
$$
a contradiction for $\epsilon$ sufficiently small.

Therefore, $I$ must be empty, and we have 

\begin{equation} \label{eqn:vanishsubsum}
z_1 + z_2 + \cdots + z_{\ell(n)} = 0.
\end{equation}

Arguing as above, for some $\delta_2 >0$ the interval $(\ctwo n, 2\ctwo n]$ contains a gap of size $\delta_2 n$ with no numbers of the form $\| \alpha \|_{\rho}$ with $\alpha$ $n$-residual.  Say this gap starts at $\secondgap \in (\ctwo n, 2\ctwo n]$, and write 
$$
z_{\ell(n) + 1} + \cdots + z_{m(n)} = -\sum_{\firstgap < \| \alpha \|_{\rho} \leq \secondgap} \zeta_n(\alpha) \rho^{\alpha},
$$
where the $z_{\ell(n) + 1} + \cdots + z_{m(n)}$ are the monomials remaining after removal of a maximal vanishing subsum.  This sum is non-empty by Lemma~\ref{lem:goodKs}, non-degenerate by construction, and contains at most $L$ terms.

We then have by Theorem~\ref{thm:Evertse} that for any $\epsilon > 0$ and odd $n \in \cvgts$ sufficiently large,
$$
|z_1 + \cdots + z_{m(n)}|^2 = |z_{\ell(n) + 1} + \cdots + z_{m(n)}|^2 \gg \frac{e^{-2\secondgap}}{R^{c_3 L \secondgap \epsilon}},
$$
where the first equality holds by Equation~\ref{eqn:vanishsubsum}.
On the other hand
$$
|z_1 + \cdots + z_{m(n)}|^2 = \big| \sum_{\| \alpha \|_{\rho} > \secondgap + \delta_2 n} \zeta_n(\alpha) \rho^{\alpha} \big|^2 \ll e^{-2\secondgap - \delta_2 n}.
$$
Thus 
$$
e^{2 \delta_2 n} \ll R^{2 c_3 L \ctwo n \epsilon},
$$
a contradiction for $\epsilon$ sufficiently small.
\end{proof}

\proofof{Theorem of~\ref{thm:mainb}} 

Recall that $\rhoconst > \max_{1\le i\le \dim}\{\log(|\rho_i|^{-1})\}$, so that
$$
\| \alpha \|_{\rho} \leq \rhoconst {\operatorname{deg} \alpha}
$$ 
for every $\dim$-tuple $\alpha$ of non-negative integers, with $\| \alpha \|_{\rho}$ as in Equation~\ref{eqn:weightednorm}.  As noted in the discussion preceding Lemma~\ref{lem:goodKs}, we may choose constants $\cone, \ctwo,$ and $L$ such that $\cone > \max \{c_1,  (1+2\rhoconst \dim) \log R\}$ and $\ctwo > 2c_1\cone$, with $c_1$ as in Corollary~\ref{cor:weightednobiggaps} and $R$ as in Lemma~\ref{lem:heightbound}, so that the sum
$$
\sum_{\| \alpha \|_{\rho} \leq \ctwo n} \zeta_n(\alpha) \rho^{\alpha}
$$
will have at most $L$ non-zero terms for any sufficiently large $n \in \cvgts$.  Given such an $n$, we choose $\firstgap$ and $\secondgap$ to satisfy $\cone n \leq \firstgap < 2 \cone n$ and $\ctwo n \leq \secondgap,$ so that Lemma~\ref{lem:goodKs} applies when $n$ is odd.  As in Equations~\ref{eq:x1x2} and~\ref{eqn:monomials}, we write
$$
z_1 + z_2 = Q_n (\Omega - \Omega_n)
$$
and
$$
z_3 + z_4 + \cdots + z_{\ell(n)} = -\sum_{\| \alpha \|_{\rho} \leq \firstgap} \zeta_n(\alpha) \rho^{\alpha}.
$$

By Theorem~\ref{thm:Evertse} and Lemma~\ref{lem:finalnonvanishing}, we have for any $\epsilon > 0$ and sufficiently large odd $n \in \cvgts$ that
$$
|z_1 + \cdots + z_{\ell(n)}|^2 \gg \frac{1}{R^{\rhoconst L\firstgap \epsilon}R^{2\dim \rhoconst n}}
$$
while also 
$$
|z_1 + \cdots + z_{\ell(n)}|^2 \ll e^{-2\firstgap - 2\delta_1 n},
$$
so that
$$
2^{2\cone n + 2\delta_1 n} \leq e^{2\firstgap + 2\delta_1 n} \ll R^{(2 \rhoconst L \cone \epsilon + 2 \rhoconst \dim)n}.
$$
However, we have chosen $\cone$ so that $\cone > (1+2\rhoconst \dim) \log R,$ so
$$
2^{2 \cone} \geq R^{1+2 \rhoconst \dim}.
$$
Thus for any choice of $\epsilon$ satisfying $2 \rhoconst L \cone \epsilon < 1$, we obtain a contradiction for $n$ sufficiently large.
\qed

\section{Conclusion and an example}
\label{sec:conclusion}
To conclude, let us explain how the results from the preceding sections of this paper suffice to guarantee existence of matrices $A\in \mathrm{SL}_\dim(\Z)$ for which the birational map $f = g\circ h_A\colon\bP^\dim\tto\bP^\dim$ has transcendental first dynamical degree.  We do this first for general $\dim\in\Z$, using a matrix $A$ that is far from explicit.  Then we give a particular and completely explicit example in dimension $\dim=3$.

\subsection{Proof of Theorem~\ref{thm:mainapp}}
\label{sec:general}

To find an appropriate matrix $A$ for Theorem~\ref{thm:mainapp}, we begin by identifying a suitable characteristic polynomial.  Given the results from previous sections, this will be the main step.

\begin{proposition}
\label{prop:pickpoly}
For any integer $\dim\geq 3$, there exists a monic, degree $\dim$, irreducible polynomial $P\in\Z[t]$ such that
\begin{enumerate}
 \item $P(0) = 1$;
 \item The Galois group of $P$ is the full symmetric group on the roots of $P$;
 \item $P$ has at most one real root.
 \item The dominant roots of $P$ are a complex conjugate pair $\maxeval$, $\bmaxeval$.
\end{enumerate}
\end{proposition}

\proof
We begin by choosing three $\deg \dim$ monic polynomials $P_0(t),P_1(t),P_2(t)\in \Z[t]$.  Specifically, we take $P_0(t)$ to be any $\deg \dim$ polynomial that is irreducible $\mymod 2$.  Necessarily $P_0(0)\equiv 1\mymod 2$.  We then take $P_1(t) = \tilde P_1(t)(t-b)$ where $\tilde P_1$ is irreducible $\mymod 3$ and of degree $\dim -1$, and $b$ satisfies $b\tilde P_1(0)\equiv 1\mymod 3$.  

The choice of $P_2(t)$ is a bit more elaborate.  Let 
$p\equiv 1\mymod 4$ be a prime larger than $2\dim$.  In particular, $-1$ is not a quadratic residue $\mymod p$.  Choose $a,b\in \{0,\dots,p-1\}$ such that $a$ is not a quadratic residue and $b$ is a $\mymod p$ multiplicative inverse of $a(-1)^{\dim -2}((\dim -3)!)^2$. Since $\pm 1$ are both quadratic residues, whereas $a$ is not, it follows that $b$ is not a quadratic residue either.  Then
$P_2(t) = (t^2-a)(t-b)\prod_{i=1}^{d-2}(t-i^2)$ has $d-1$ distinct roots $b,1^2,\dots,(d-2)^2$ and a quadratic factor $t^2 -a$ that is irreducible $\mymod p$.

Next we apply the Chinese Remainder Theorem to obtain a polynomial $P\in\Z[t]$ such that $P\equiv P_0\mymod 2$, $P\equiv P_1\mymod 3$ and $P\equiv P_2 \mymod p$.  We may further assume $P(0) = 1$.  Then $P$, like $P_0$, is irreducible $\mymod 2$ and therefore irreducible over $\Z$.  Hence the Galois group of $P$ is transitive.  Dedekind's Theorem (cf. Lang \cite[Theorem 2.9, p. 345]{Lang}) and $P\equiv P_1\mymod 3$ implies that the Galois group contains a $d-1$ cycle.  Likewise, $P\equiv P_2\mymod p$ implies that the Galois group contains a transposition.  Standard theory of permutation groups tells us that a transitive subgroup of the symmetric group on $\dim$ elements is the full group as soon as it contains a transposition and a $\dim-1$-cycle.  Thus $P$ satisfies the first two conclusions of the proposition.

To guarantee it also satisfies the last two conclusions, we will replace $P$ by $P+Q$ for some polynomial $Q\in\Z[t]$ satisfying $Q(0) = 0$ and $Q\equiv 0 \mymod 6p$.  

\begin{lemma}
\label{lem:onerealroot}
Let $P(t)$ be a monic real polynomial with $\deg P = \dim\geq 3$ and $P(0) = 1$.  Then for large enough $a,b>0$ the following hold.
\begin{itemize}
 \item If $\dim$ is even, and $Q(t) := a t^{\dim-2} + b t^2$, then $P+Q$ has no real roots. 
 \item If $\dim$ is odd, and $Q(t) := a t^{\dim-2} + b t$, then $P+Q$ has exactly one real root.
\end{itemize}
\end{lemma}

\begin{proof}
Suppose first that $\dim$ is even. Then $P(t) = t^\dim + c_{\dim-1} t^{\dim -1} + t^2 R(t) + c_1 t + 1$, where $c_{n-1},c_1\in\R$ and $\deg R(t) \leq \dim -4$.  Hence 
$$
P(t)+Q(t) = t^{\dim-2}\left(t^2 + c_{\dim-1} t + \frac{a}{2}\right) + t^2\left(R(t) + \frac{a}2 t^{\dim-4} + \frac{b}2\right)
+ \left(\frac{b}2 t^2 + c_1 t + 1\right), 
$$
and one checks easily that all three polynomials in parentheses are positive for $a,b>0$ large enough and any $t\in\R$.
Hence $P+Q$ has no real roots.

When $\dim$ is odd one checks by the same sort of computation that $Q'(t)>0$ for all $t\in\R$ when $a,b>0$ are large, so in this case $P+Q$ has exactly one real root.
\end{proof}

We can now conclude the proof of Proposition~\ref{prop:pickpoly} as follows.  Assume that $\dim$ is even and let $a,b\in\Z_{\geq 0}$ be positive multiples of $6p$ chosen large enough that Lemma~\ref{lem:onerealroot} holds.
Then $P(t) + Q(t) = P(t) + a t^{\dim -2} + b t^2$ satisfies the first three conclusions of the proposition.  Note that
$P(t) + Q(t) = t^{\dim} + at^{\dim - 2} + \tilde P(t)$, where $\deg \tilde P(t) = \dim -1$ does not depend on $a$.  Let
$M = \max_{|t|=1} |\tilde P(t)|$.  Then increasing $a$ if necessary, we have
$$
|t^{\dim} + a t^{\dim-2}| \geq a-1 > M
$$
whenever $|t|=1$.  Hence by Rouch\'e's Theorem, $P+Q$ and $t^{\dim} + a t^{\dim -2}$ have the same number of zeroes in the unit disk, i.e.\ $\dim - 2$ of them.  As $P+Q$ has no real roots, the two roots outside the unit disk are a complex conjugate pair.  So all four conclusions of the proposition hold with $P+Q$ in place of $P$.  The case when $\dim$ is odd is similar, and we leave it to the reader.  
\qed

\begin{proposition}
\label{prop:irresonant}
When $\dim\geq 3$, there are no angular resonances between distinct roots of the polynomial $P$ in Proposition~\ref{prop:pickpoly}.
\end{proposition}

\begin{proof} 
If $z,w\in\C$ are distinct roots of $P$ with an angular resonance $z^a\bar w^b \in \R$ for some $a,b>0$, then
$$
z^a\bar w^b = \bar z^a w^b
$$
is a relationship between four roots of $P$.  Since $P$ has at most one real root, we may assume that at least $z$ is not real.  Assume that $z\neq\bar w$ (the case $z=\bar w$ is similar).  Since the Galois group of $P$ is the full symmetric group on the roots, it includes the transposition exchanging $w$ and $\bar w$.  Applying it gives the additional relation $z^a w^b = \bar z^a \bar w^b$.  Multiplying our two relations, we infer $z^{2a} = \bar z^{2a}$.  But now we can use the Galois group to exchange $\bar z$ with any root $z'$ distinct from $z$ to obtain that $z^{2a} = (z')^{2a}$ for all roots $z'$ of $P$.  In particular, all roots of $P$ have the same magnitude.  When $\dim \geq 3$, this contradicts that $P$ has exactly two roots $\maxeval,\bmaxeval$ of largest magnitude.
\end{proof}

To complete the proof of Theorem~\ref{thm:mainapp}, we let $\tilde A\in \mats_\dim(\Z)$ to be the companion matrix of the degree $d$ polynomial $P\in\Z[x]$ from Proposition~\ref{prop:pickpoly}.  Then Propositions~\ref{prop:pickpoly} and~\ref{prop:irresonant} tell us that $\tilde A$ satisfies all the conditions of Theorem~\ref{thm:main}; in particular $P(0) =1$ means that $A\in \SL_\dim(\Z)$.  The transcendence statement in the conclusion of Theorem~\ref{thm:main} therefore holds for an appropriate conjugate $A = Y\tilde A Y^{-1}$ of $\tilde A$.  On the other hand, since there are no angular resonances between roots of $P$, the leading eigenvalue $\maxeval$ of $A$ satisfies $\maxeval^j\notin\R$ for any positive integer $j$.  Hence Theorem~\ref{thm:degformula} tells us that for a sufficiently high power $A^N$, the dynamical degree $\ddeg(f)$ of $f := \inv\circ\momap_{A^N}$ satisfies
$$
1 = \sum_{n=1}^\infty \fnal_{\rvecs,\cvecs}(\op^{Nn})\ddeg(f)^{-n}.
$$
Returning to Theorem~\ref{thm:main}, we infer that $x=\ddeg(f)^{-1}$ is not algebraic, and our main result Theorem~\ref{thm:mainapp} is proved.

\subsection{A specific example}
\label{sec:specific}
With some computer assistance, one can also verify that Theorem~\ref{thm:mainapp} holds for specific, explicit choices of the matrix $A\in\SL_\dim(\Z)$.  We illustrate this in dimension $\dim = 3$, starting with the companion matrix
$$
\tilde A = \begin{pmatrix} 
0 & -1 & 1 \\
1 & 0 & 0 \\
0 & 1 & 0
\end{pmatrix}\in \mathrm{SL}_3(Z)
$$
for the polynomial $P(t) = t^3 - t+1$.  Since $P(0) = P(1) = 1$, one sees that $P$ is irreducible $\mymod 2$ and therefore irreducible over $\Z$.  The leading roots of $P$ are a conjugate pair $\maxeval,\bmaxeval$ where $\maxeval \approx -0.341164 + 1.16154 i$, and the remaining root is real equal to $|\maxeval|^{-2}<1$.  Moreover, by computing its minimal polynomial one checks that $\maxeval/\bar\maxeval$ is not a root of unity.  Hence $\maxeval^j\notin\R$ for any non-zero $j\in\Z$, and as we noted following Theorem~\ref{thm:main}, this implies there are no angular resonances among the roots of $P$.  All told, these facts allow us to apply Theorems~\ref{thm:main} and~\ref{thm:degformula} to $\tilde A$ as above.  

We claim in fact that taking 
$$
Y = \begin{pmatrix} 1 & -2 & 3 \\ 0 & 1 & -2 \\ 0 & 0 & 1 \end{pmatrix}
$$
in Theorem~\ref{thm:main} and then $N=7$ in Theorem~\ref{thm:degformula} suffices; i.e.\ Theorem~\ref{thm:mainapp} holds with
\begin{equation}
\label{eqn:eg}
A = Y\tilde A^7 Y^{-1} = 
\begin{pmatrix}
-3 & -14 & -12 \\ 4 & 11 & 6 \\ -2 & -4 & -1
\end{pmatrix}.
\end{equation}
To justify this, one needs to verify two things:
\begin{itemize}
 \item the function $\gamma = \gamma_A$ constructed in \S\ref{sec:prelim} satisfies the discordance condition in Theorem~\ref{thm:mainb}; and
 \item $A$ satisfies the hypothesis of Theorem~\ref{thm:degformula2}.
\end{itemize}
Accomplishing the first task is straightforward and can be achieved even for $N=1$, i.e.\ for $Y\tilde AY^{-1}$ in place of $A$.  Equation~\eqref{eqn:candidatedirection} tells us that the discontinuity set $\discty(\gamma_A)$ consists of normalized arguments of finitely many elements $\rfn(Y,\cvec,\dvec)$ of the splitting field $K$ for $\maxeval$, one for every pair of vectors $\cvec\in\cvecs$, $\dvec\in\dvecs$.  Even without accounting for repetition, this yields less than fifty possible elements of $K$.  It suffices (see Remark~\ref{rmk:practical}) to verify that all of them, together with all of their non-trivial ratios, lie outside the ring of units $\cO_K^*$.  Standard computer algebra packages do this easily.

The second task is harder.  To verify the hypothesis of Theorem~\ref{thm:degformula2} it suffices to show that $8 = 2(3 + 1)$ vectors $\cvec\in\Z^3$ have strict forward orbits $(A^n\cvec)_{n\geq 1}$ that avoid $6 = \frac{3(3 + 1)}{2}$ rational two-dimensional hyperplanes in $\Z^3$.  This boils down to showing that $48$ integer linear recurrences $(a_n)_{n\in\Z_{\geq 0}}$ have no zeroes beyond the initial term $a_0$.  For this we show that in our situation, the Skolem-Mahler-Lech Theorem can be made more effective as follows.

\begin{lemma}
None of the linear recurrence sequences $(a_n)_{n\geq 0}$ of interest have vanishing terms $a_n$ with $n\geq 10^{20}$.
\end{lemma}

\begin{proof}
We only sketch the argument.  The terms in any linear recurrence of interest here have the form
$$
a_n = \sum_{j=1}^3 c_j \eval_j^n
$$
where $\eval_1 = \maxeval$, $\eval_2 =\bmaxeval$ and $\eval_3 = |\maxeval|^{-2}$ are the eigenvalues of $A$, and $c_j\in K$ are determined by $A$ and a choice of $\cvec \in \cvecs$ and $\dvec\in\dvecs$.  So if $a_n =0$, we obtain an exponential (in $n$) upper bound
$$
|a(\maxeval/\bmaxeval)^n - 1| \leq b|\maxeval|^{-3n},
$$
which is equivalent to 
$$
|\log a + n\log(\maxeval/\bmaxeval)| \leq b c^{-n}
$$
for some (explicit) constants $a\in K$ and $b,c\geq 1$.  The expression inside absolute values on the left is a linear form in logarithms with integer coefficients.  Hence a result of Baker and Wusth\"oltz~\cite{BaWu93} gives a lower bound for the same quantity of the form $b' n^{-c'}$ where the constants $b',c'>0$ are again explicit and derived from $A$, $\cvecs$ and $\dvecs$.  Since the Baker--Wusth\"oltz bound is polynomial in $n^{-1}$, it is inconsistent with the exponentially decaying upper bound for large $n$.  Carefully tracking all constants, one finds that if $a_n = 0$, then $n$ must be smaller than $10^{20}$.
\end{proof}

It remains to verify that none of the first $10^{20}$ terms vanish in each of the linear recurrences $(a_n)$.  This is impractical to do directly even with computer assistance.  However, one can avoid direct verification by reducing the recurrences modulo various primes $p$.  The advantage is that modulo $p$, all the recurrences become periodic with period no more than e.g.\ the number of invertible $3\times 3$ matrices with coefficients in $\Z/ p\Z$.  It turns out, moreover, that there are many primes $p$ that are `good' in the sense that the sequence $(A^n\mymod p)_{n\in\Z_{\geq 0}}$ has period $m_p$ dividing $p-1$.  For such $p$ it is often the case that there is no more than one vanishing term $a_n\mymod p$ among the first $m_p$.  

If we find a specific prime for which \emph{no} terms $a_n\mymod p$ of the reduced recurrence vanish, we are done.  In our example this happens for more than half the recurrences we consider.  For all but one of the other recurrences, the initial term $a_n$ term of the unreduced recurrence vanishes, so it must be the case that $a_n\mymod p$ vanishes for all $n\equiv 0\mymod m_p$ in any reduction, too.  However, for many good primes $p$, the initial term is the \emph{only} one of the first $m_p$ terms whose reduction vanishes.  It follows that the smallest positive index $n$ for which the (unreduced) term $a_n$ vanishes is at least as large as the product of the periods $m_p$ associated to these good primes.  With some computer algebra one easily finds enough good primes to boost the product past $10^{20}$.  

In our example, there is only one recurrence $(a_n)$ not covered by either of these considerations: i.e.\ $a_0 \neq 0$ but $(a_n\mymod p)_{n=0}^{m_p-1}$ seems to always include at least one vanishing term.  Nevertheless, by focusing on those good primes $p$ for which only one reduced term vanishes among the first $m_p$, one can use the Chinese remainder theorem to synthesize the information from reductions by various good primes and get a lower bound on the index of the first vanishing term in the reduced recurrence.  Computer algebra again allows one to boost the bound past $10^{20}$ without much trouble and complete the verification that $A$ satisfies the hypothesis of Theorem~\ref{thm:degformula2}.

\end{document}